\documentclass[12pt,A4]{article}
\usepackage{amsmath}
\usepackage{amssymb}
\usepackage{amscd}
\usepackage{amsthm} 
\usepackage{amsopn}
\usepackage{xspace}
\usepackage{euscript}
\usepackage{verbatim}
\usepackage{amsfonts}
\usepackage{epic}
\usepackage{eepic}
\usepackage{subfigure}
\usepackage{empheq}
\usepackage[active]{srcltx}
  \newtheorem{theorem}{Theorem}
  \newtheorem{lemma}{Lemma}[section]
  \newtheorem{proposition}{Proposition}[section]
  \newtheorem{corollary}{Corollary}[section]

\numberwithin{equation}{section}

\newcommand{\be}{\begin{equation}}
\newcommand{\ee}{\end{equation}}
\newcommand{\bd}{\begin{displaymath}}
\newcommand{\ed}{\end{displaymath}}

\newcommand{\N}{\mathbb N}
\newcommand{\R}{\mathbb R}
\newcommand{\C}{\mathbb C}

\newcommand{\A}{\mathcal A}

\newcommand{\LL}{\mathcal L}
\newcommand{\OO}{\mathcal O}

\newcommand{\Eb}{\mathbb E}
\newcommand{\Cb}{\mathbf C}

\newcommand{\J}{\mathcal J}
\newcommand{\HH}{\mathcal H}
\newcommand{\W}{\mathcal W}

\newcommand{\Union}{\mathop{\bigcup}\limits}

\newcommand{\Prod}{\mathop{\prod}\limits}

\DeclareMathOperator{\Div}{div}

\def\XXint#1#2#3{{\setbox0=\hbox{$#1{#2#3}{\int}$ }
\vcenter{\hbox{$#2#3$ }}\kern-.6\wd0}}

\def\Ig {{\mathcal I}}

\def\Kg {{\mathcal K}}

\def\Ug {{\mathcal U}}

\def\Vg {{\mathcal V}}

\begin{document}
\bibliographystyle{siam} \title{The Clausius-Mossotti formula for
  dilute random media of perfectly conducting inclusions} \author{Y.
  ALMOG \thanks{Department of Mathematics, Louisiana State University,
    Baton Rouge, LA 70803, USA},} \date{}
\maketitle

\begin{abstract}
  We consider a large number of randomly dispersed spherical,
  identical, perfectly conducting inclusions (of infinite
  conductivity) in a bounded domain. The host medium's conductivity is
  finite and can be inhomogeneous.  In the dilute limit, with some
  boundedness assumption on a large number (proportional to the global
  volume fraction raised to the power of $-1/2$) of marginal
  probability densities, we prove convergence in $H^1$ norm of the
  expectation of the solution of the steady state heat equation, to
  the solution of an effective medium problem, where the conductivity
  is given by the Clausius-Mossotti formula. Error estimates are
  provided as well.
\end{abstract}

\section{Introduction}

Consider a $N$ spherical perfectly conducting inclusions of radius
$\epsilon$ immersed in a different medium of non-uniform conductivity $a$.
Prescribing the temperature (or the electric potential) on the
boundary, the temperature field inside can be described as the unique
solution of the problem
\begin{subequations}
  \label{eq:1}
\begin{empheq}[left={\empheqlbrace}]{alignat=2}  
  &  \Div(a\nabla\phi)  = 0 \quad & \text{in } & \Omega\setminus\Union_{n=1}^N B(\eta_n,\epsilon) \,, \\
  & \phi= f \quad & \text{on } & \partial\Omega \,, \\
  &  \phi=C_n \quad &\text{in } & B(\eta_n,\epsilon)\,,\;1\leq n\leq N \,, \\
  &  \int_{\partial B(\eta_n,\epsilon)} a\frac{\partial\phi}{\partial\nu}\ \,ds =0 \,. \quad &\quad & \quad
  \end{empheq}
\end{subequations}
In the above, $\Omega\subset\R^3$ is bounded and smooth (say $C^{2,\alpha}$ for
some positive $\alpha\geq1/2$), $a\in C^{1,\alpha}(\bar{\Omega},\R_+)$, and hence,
\begin{equation}
\label{eq:2}
  0<\lambda\leq a(x)\leq\Lambda \quad\forall x\in\bar{\Omega}\,,
\end{equation}
$\{\eta_i\}_{i=1}^N$ denote the spherical
inclusion centers, and $f\in C^{2,\alpha}(\partial\Omega)$.

The particles' centers are assumed to be randomly distributed
according to the joint probability density function
$f_N(\eta_1,\ldots,\eta_N)$, which is assumed to be invariant to permutations
of the centers as all particles are identical. Moreover, we assume
that the inclusions cannot overlap, i.e., 
\begin{equation}
  \label{eq:3}
 \exists1\leq i<j\leq N:\;|\eta_i-\eta_j|<2\epsilon\Rightarrow f_N(\eta_1,\ldots,\eta_N)=0 \,,
\end{equation}
and that no inclusion can cross the boundary, i.e.,
\begin{equation}
\label{eq:4}
  \exists1\leq i\leq N: \,d(\eta_i,\partial\Omega)< \epsilon \Rightarrow f_N(\eta_1,\ldots,\eta_N)=0 \,.
\end{equation}

We focus our attention on the small particle limit in a dilute  (or
dispersive \cite{ko89} ) medium, i.e., we first let $\epsilon\to0$ but keep
the volume fraction $\bar{\beta}$ fixed, where
\begin{equation}
\label{eq:5}
\bar{\beta} = \frac{4\pi}{3} \frac{N\epsilon^3}{|\Omega|}\,,
\end{equation}
and then let $\bar{\beta}\to0$.
Note that $N$ must tend to infinity as $\epsilon\to0$ when $\bar{\beta}$ is
fixed. As in \cite{al13} we
assume
\begin{equation}
\label{eq:6}
  \frac{\epsilon}{C} <\bar{\beta} \leq \frac{C}{\ln^4 \epsilon^{-1}} \,.
\end{equation}

Let
\begin{displaymath}
  f_k(\eta_1, \ldots,\eta_k) = \int_{\Omega^{N-k}} f_N (\eta_1, \ldots,\eta_k,\eta_{k+1},
  \ldots,\eta_N) \, d\eta_{k+1} \cdots d\eta_N \,,
\end{displaymath}
denote the $k'th$ order marginal probability density. We assume here
that for some $C_0>0$
\begin{equation}
  \label{eq:7}
\|f_k\|_{L^\infty(\Omega^k)} \leq C_0^k \quad \forall 1\leq k\leq\bar{\beta}^{-1/2} \,,
\end{equation}
where $C_0$ is independent of $N$ and $\epsilon$. We denote the expectation of
any function $F(x,\cdot)\in L^1(\Omega^N)$, where $x\in\Omega$, by
\begin{equation}
  \label{eq:8}
\Eb_f\big(F(x,\cdot)\big) = \int_{\Omega^N} F(x,\eta_1,\ldots,\eta_N)  f_N(\eta_1,\ldots,\eta_N)\, d\eta_1 \cdots d\eta_N \,.
\end{equation}

Define, next, the local volume fraction for all $x\in\Omega$
\begin{equation}
\label{eq:9}
  \beta(x) = N \int_{B(x,\epsilon)\cap\Omega_\epsilon}f_1(\eta) \,d\eta \,,
\end{equation}
where
\begin{equation}
\label{eq:10}
  \Omega_\epsilon = \{x\in\Omega \,| \, d(x,\partial\Omega)>\epsilon \} \,.
\end{equation}
Note that $\beta(x)$ is the the probability that
$x\in\Union_{n=1}^NB(\eta_n,\epsilon)$. It follows from \eqref{eq:7} that
\begin{equation}
\label{eq:11}
  \|\beta(\cdot)\|_\infty \leq C \bar{\beta} \,.
\end{equation}
Where $\|\cdot\|_p$ denotes the $L^p(\Omega)$ norm ($p=\infty$ above). When $L^p$
norms are evaluated over domains different than $\Omega$, we shall
include them explicitly in the notation.

Under the above assumptions we prove the following theorem
\begin{theorem}
  \label{thm:1}
  Let $\phi(\cdot,\eta_1,\ldots,\eta_N)\in H^1(\Omega)$ denote the unique weak solution
  of \eqref{eq:1}, and suppose that \eqref{eq:7} is
  satisfied. Let $\phi_e$ denote the solution of the effective medium
  problem
\begin{equation}
\label{eq:12}
  \begin{cases}
    \nabla\cdot(a_e\nabla\phi_e) = 0 & \text{in } \Omega \\
    \phi_e =f & \text{on }\partial\Omega \,,
  \end{cases}
\end{equation}
where 
\begin{equation}
\label{eq:13}
  a_e(x) =a(x)( 1  + 3\beta(x))\,.
\end{equation}
Then,  in the regime of \eqref{eq:6},  we have 
\begin{equation}
\label{eq:14}
\|\Eb_f(\phi) -\phi_e\|_{1,2} \leq C(\Omega,\sigma)\bar{\beta}^{5/4} \,,
\end{equation}
where $\|\cdot\|_{1,p}$ denotes the $W^{1,p}(\Omega)$ norm.
\end{theorem}
Throughout the sequel,  we always refer to solutions in a weak sense,
including places in the text where we do not state that
explicitly.  

In \cite{al13,al14} results similar Theorem \ref{thm:1} have been
obtained, for the case where the conductivity of the inclusion is a
fixed and strictly positive. Within the theory of stochastic
homogenization, a similar situation has been treated in \cite{mo13} for
a discrete operator, whereas in \cite{dugl16} the continuous case has
been addressed, including the vector case (thereby establishing a
Clausius-Mossotti formula for the elastic constants of random
composite media). Despite the greater generality of the results in
\cite{dugl16} from some aspects (it also applies to dilute random
perturbation of random media) , the technique in \cite{al13,al14} as
well as in the present contribution does not assume stationarity and
ergodicity of the probability density. These are fundamental
assumptions in the theory of homogenization \cite{ko79,pava81}, and
are certainly assumed in \cite{dugl16}. Moreover, the results in
\cite{al13,al14} can be easily generalized to higher dimensions and to
arbitrary inclusion shape (and even random shapes). We skip these
generalization here for the sake of simplicity, but manifest the
greater generality of our technique by allowing for an inhomogeneous
conductivity. We note that one can extend the present analysis to the
case where the conductivity is anisotropic (or when $a$ is replaced by
a positive $3\times3$ symmetric matrix $A\in C^{1,\alpha}(\Omega)$ whose eigenvalues
satisfy \eqref{eq:2}), as the Green's function estimate in Appendix A
apply to this case as well. The vector case is deferred to a later
stage.

The main progress offered by the present contribution is that it
addresses the case of perfect conductors (of infinite conductivity).
This requirement of uniform ellipticity, is a standard assumption
within the theory of homogenization, and is certainly assumed in
\cite{dugl16}. Whereas for finite conductivity, one can present the
effect of inclusions by a discontinuous conductivity function, in the
case of perfect conductors it is impossible. Instead, we use here
variational techniques to estimate the error generated by
approximating the contribution of each inclusion via the assumption
that it is given in a homogeneous temperature gradient field. These
variational estimates require further assumptions in the form of
\eqref{eq:7} beyond those made in \cite{al13}. Certainly, one cannot
obtain from them the $L^p$ (for $1<p<2$) estimates derived in
\cite{al14}. We note that the homogenization of  a dilute periodic array of
perfect conductors for the time dependent heat equation has been
treated in \cite{mipr92}. In \cite{deetal14} the   time dependent
problem is addressed for a random medium, assuming that
$|\eta_i-\eta_j|\geq C\bar{\beta}^{1/3}$ for all $t>0$.

The rest of this contribution is  arranged as follows. In the next
section we review a few basic preliminaries. In sections 3 and 4 we
respectively derive (as in \cite{al13}) a few necessary inequalities
for the solution of \eqref{eq:1} with $N=1$ and $N=2$. In \S\,5 we
derive an estimate of $\Eb_f(\phi)$ using the  single single
inclusion solution of \S\,3. Finally, in \S\,6 we complete the proof of Theorem
\ref{thm:1}. 

http://www.zemereshet.co.il/song.asp?id=3704
\section{Preliminaries}

\subsection{A variational principle}
Set $B_n=B(\eta_n,\epsilon)$, and ${\mathbf C}=(C_1,\ldots,C_N)\in\R^N$.
\begin{lemma}
Let 
\begin{equation}
  \label{eq:15} I_N(w) = \int_{\Omega\setminus\Union_{n=1}^NB_n}a|\nabla w|^2\,dx,
\end{equation} 
defined on 
\begin{subequations}
  \label{eq:16}
\begin{equation}
X_N(f,\{g_n\}_{n=1}^N)=\Union_{\Cb\in\R^N}\HH_N(\Cb,f,\{g_n\}_{n=1}^N)\,,
\end{equation}
where
\begin{equation}
\HH_N(\Cb,f,\{g_n\}_{n=1}^N)=\{w\in H^1(\Omega)\,|\,w|_{\partial\Omega}=f \,, \; w|_{\partial B_n}=g_n+C_n
; 1\leq n\leq N \}\,,
\end{equation}
  \end{subequations}
and $f$ and $\{g_n\}_{n=1}^N$ are in $C^{2,\alpha}$.  There exists a
unique minimizer for $I_N$ in $X_N$. Furthermore, the minimizer must
be the unique solution of
\begin{subequations}
\label{eq:17}
\begin{empheq}[left={\empheqlbrace}]{alignat=2}  
  &  \Div(a\nabla v)  = 0 & \text{in } \Omega\setminus\Union_{n=1}^N B_n \,, \\
   & v= f & \text{on } \partial\Omega \,, \\
  &  v=C_n+g_n & \text{in } B_n\,,\;1\leq n\leq N \,, \\
  &  \int_{\partial B_n}a \frac{\partial v}{\partial\nu}\ \,ds =0 \,. &
  \end{empheq}
\end{subequations}
\end{lemma}
We skip here the rather standard proof of this lemma.

\subsection{An integral representation}
\label{sec:integralrep}
Let $\bar{\phi}$ denote the unique solution of
\begin{equation}
\label{eq:18}
  \begin{cases}
    \LL\bar{\phi} = 0 & \text{in } \Omega \\
    \bar{\phi}=f & \text{on } \partial\Omega\,,
  \end{cases}
\end{equation}
where $\LL\overset{def}{=}-\Div a\nabla$. Let $G:\Omega\times\Omega\to\R_+$ denote the
Green's function associated with the Dirichlet realization of $\LL$ in
$\Omega$. Then, we have by Green's formula, for all
$x\in\Omega\setminus\Union_{n=1}^NB_n$,
\begin{multline*}
  \phi(x,\eta_1,\ldots,\eta_N) = \int_{\partial\Omega\cup\Union_{n=1}^N\partial B_n^{out}}a(\xi)\Big[ G(x,\xi)\frac{\partial\phi}{\partial\nu}(\xi,\eta_1,\ldots,\eta_N)
  \\-\phi(\xi,\eta_1,\ldots,\eta_N)\frac{\partial G}{\partial\nu} (x,\xi)\Big] ds_\xi \,.
\end{multline*}
Since $G(x,\cdot)|_{\partial\Omega}=0$, and
\begin{displaymath}
    \bar{\phi} =  - \int_{\partial\Omega}a
  \phi\frac{\partial G}{\partial\nu}\,ds_\xi \,,
\end{displaymath}
we obtain that
\begin{multline*}
  \phi(x,\eta_1,\ldots,\eta_N)=\bar{\phi}(x) + \sum_{n=1}^N\int_{\partial B_n^{out}} \Big[ G(x,\xi)a(\xi)\frac{\partial\phi}{\partial\nu}(\xi,\eta_1,\ldots,\eta_N)
  -\\\phi(\xi,\eta_1,\ldots,\eta_N)a(\xi)\frac{\partial G}{\partial\nu} (x,\xi)\Big]\,ds_\xi \,.
\end{multline*}
Using (\ref{eq:1}c) and the fact that $\LL G=0$ in $B_n$ for all
$1\leq n\leq N$ then yields
\begin{equation}
  \label{eq:19}
\phi(x,\eta_1,\ldots,\eta_N)=\bar{\phi}(x) + \sum_{n=1}^N\int_{\partial B_n^{out}} 
G(x,\xi)a(\xi)\frac{\partial\phi}{\partial\nu}(\xi,\eta_1,\ldots,\eta_N) \,ds_\xi \,,
\end{equation}
for all $x\in\Omega\setminus\Union_{n=1}^NB_n$.

\subsection{N-particle capacity}
\label{sec:Ncapacity}

Let
\begin{equation}
\label{eq:20}
  C_K(\eta_1,\ldots,\eta_K) = \inf_{u\in\W(\eta_1,\ldots,\eta_K)} \|\nabla
  u\|_2^2 \,,
\end{equation}
where
\begin{displaymath}
  \W(\eta_1,\ldots,\eta_K) =\big\{u\in H^1_0(\Omega)\,|\,
  u\big|_{\Union_{k=1}^KB_k}\equiv1\, \big\} \,.
\end{displaymath}
We begin with the following useful bound
\begin{lemma}
\label{lem:standardbound}
  Let $C_K(\eta_1,\ldots,\eta_K)$ be defined by \eqref{eq:20}. Then
  \begin{equation}
    \label{eq:21}
C_K(\eta_1,\ldots,\eta_K) \leq \sum_{k=1}^K C_1(\eta_k) \,.
  \end{equation}
\end{lemma}
See \cite[Proposition 4.1.3]{ma85} for the proof.  

We can now establish the following bound for $C_K$
\begin{lemma}
  Let
  \begin{equation}
\label{eq:22}
    \delta_k=\min\Big(1,\frac{d(\eta_k,\partial\Omega)}{\epsilon}-1\Big)
  \end{equation}
and $C_K(\eta_1,\ldots,\eta_K)$ be given by
  \eqref{eq:20}. Then, there exists $C>0$, such that  
  \begin{equation}
    \label{eq:23}
C_K(\eta_1,\ldots,\eta_K)\leq C \epsilon \sum_{k=1}^K (1+\ln(1/\delta_k)) \,.
  \end{equation}
\end{lemma}
\begin{proof}
  In view of \eqref{eq:21} suffices it to show that
  \begin{equation}
\label{eq:24}
    C_1(\eta) \leq C \epsilon \ln(1/\delta) \,,
  \end{equation}
where $\delta=d(\eta,\partial\Omega)/\epsilon-1$
To this end let $x\in \Omega$. Let further $l_x$ denote the straight ray emanating
from  $\eta$ and passing through $x$. Let $s_x$ denote the closest (to
$x$) intersection point of $l_x$ with $\partial\Omega$. Set then
\begin{displaymath}
  \tilde{u}(x) = \chi(|x-\eta|/\epsilon)
  \begin{cases}
    1 & x\in B(\eta,\epsilon) \\
    1- \frac{|x-\eta-\epsilon|}{|s_x-\eta-\epsilon|} &x\in\Omega\setminus B(\eta,\epsilon) \,,
  \end{cases}
\end{displaymath}
The cutoff function $\chi\in C^1(\R_+;[0,1])$ satisfies
\begin{equation}
  \label{eq:25}
\chi(x) =
\begin{cases}
  1 & x\leq1 \\
  0 & x\geq2 
\end{cases}
|\chi^\prime|\leq C \,.
\end{equation}
It can be easily verified that
\begin{displaymath}
  \|\nabla\tilde{u}\|_2^2\leq C\epsilon\ln(1/\delta) \,,
\end{displaymath}
from which \eqref{eq:24}, and then \eqref{eq:23} easily follow.
\end{proof}

\section{Single inclusion}
\label{sec:single-inclusion}
We next define the one-particle problem. Let $B=B(\eta,\epsilon)$ and
\begin{equation}
\label{eq:26}
  \Omega_\epsilon = \{ x\in \Omega \,|\, d(x,\partial\Omega)>\epsilon\}\,.
\end{equation}
For every $\eta\in\Omega_\epsilon$, let
$\psi_1(\cdot,\eta):\Omega\to\R$ and $C_1\in\R$ denote the unique solution of
\begin{equation}
  \label{eq:27}
  \begin{cases}
    \LL \psi_1(\cdot,\eta) =0 & \text{in } \Omega\setminus B \\
    \psi_1(\cdot,\eta) = f & \text{ on } \partial\Omega \\
    \psi_1=C_1 & \text{in } B\,, \\
    \int_{\partial B} a \frac{\partial\psi_1}{\partial\nu} \,ds =0 \,.   
\end{cases}
\end{equation}
Set 
\begin{equation}
\label{eq:28} 
 \phi_1(\cdot,\eta):=\psi_1(\cdot,\eta)-\bar{\phi} \,.
\end{equation}

For all $\eta\in\Omega_\epsilon$, define $\phi_0(\cdot,\eta):\Omega\to\R$ as
\begin{equation}
\label{eq:29}
  \phi_0 (x,\eta)=
  \begin{cases}
-(x-\eta)\cdot\nabla\bar{\phi}(\eta) \frac{\epsilon^3}{|x-\eta|^3}+ C_a\frac{\epsilon^3}{r} &
    x\in\Omega\setminus B(\eta,\epsilon) \\
 -(x-\eta)\cdot\nabla\bar{\phi}(\eta) +C_a\epsilon^2 &
    x\in B(\eta,\epsilon)
  \end{cases}\,.
\end{equation}
Where $C_a$ is so chosen that
\begin{equation}
\label{eq:30}
  \int_{\partial B(\eta,\epsilon)}a(\xi)\frac{\partial\phi_0}{\partial\nu}(\xi,\eta)\,ds_\xi =0\,.
\end{equation}
It can be easily verified that
\begin{displaymath}
  C_a(\eta) = \frac{2}{\epsilon}\frac{\int_{
    B(\eta,\epsilon)}\nabla a(\xi)\cdot\nabla\bar{\phi}(\eta)\,d\xi }{ \int_{\partial B(\eta,\epsilon)}a(\xi)\,ds_\xi }\,.
\end{displaymath}
As $a\in C^{1,\alpha}$ we then obtain that
\begin{equation}
\label{eq:31}
  C_a = \frac{2}{3}\frac{\nabla\bar{\phi}(\eta)\cdot\nabla a(\eta)}{a(\eta)}+ \OO(\epsilon^\alpha)\,.
\end{equation}

We can now state the following
\begin{lemma}
\label{lem:v1}
  Let $\phi_1$ be given by \eqref{eq:28}. Define for each $\eta\in\Omega_\epsilon$, $u_1(\cdot,\eta):\Omega\to\R$ by
  \begin{equation}
    \label{eq:32}
v_1(x,\eta) = \phi_1(x,\eta) - \phi_0(x,\eta)\,.
  \end{equation}
Then,
\begin{equation}
  \label{eq:34}
\|\nabla v_1(\cdot,\eta)\|_2\leq  C(\Omega) \Big( \epsilon^{5/2} + \frac{\epsilon^3}{d(\eta,\partial\Omega)^{3/2}}\Big) \,.
\end{equation}
\end{lemma}
\begin{proof}
  By \eqref{eq:27}, \eqref{eq:28}, and \eqref{eq:32} we have that
  $(v_1,C_1)$ is the solution of
  \begin{equation}
\label{eq:35}
  \begin{cases}
    \LL v_1(\cdot,\eta) =0 & \text{in } \Omega\setminus B \\
    v_1(\cdot,\eta) = -\phi_0(\cdot,\eta) & \text{ on } \partial\Omega \\
    v_1=C_1-\bar{\phi}_2(\cdot,\eta) & \text{in } B\,, \\
    \int_{\partial B} a\frac{\partial v_1}{\partial\nu} \,ds =0  \,,   
\end{cases}
  \end{equation}
where
\begin{equation}
\label{eq:36} 
\bar{\phi}_2(x,\eta)= \bar{\phi}(x)- \bar{\phi}(\eta)-(x-\eta)\cdot\nabla\bar{\phi}(\eta) \,.
\end{equation}
 It can be easily verified that $v_1$ is the minimizer in
 $X_1(-\phi_0,-\bar{\phi}_2)$ of \eqref{eq:15}.

Suppose first that $d(\eta,\partial\Omega)\geq2\epsilon$. Set then
\begin{displaymath}
  w(x)= \zeta(t) \phi_0(s,\eta)(1-\chi(|x-\eta|/\epsilon)) + \chi(|x-\eta|/\epsilon)\bar{\phi}_2(x) \,,
\end{displaymath}
in which $t=d(x,\partial\Omega)$ and $s$ is the projection of $x$ on $\partial\Omega$, which is
well-defined for all $t<\delta_0$ (where  $\delta_0$ is a property of the
smooth boundary). The cutoff function $\zeta\in C^1(\R_+;[0,1])$ is 
supported on $[0,\delta]$, for some $0<\delta<\delta_0$, and satisfies 
$|\zeta^\prime|\leq C/\delta$, and $\chi$ is given by \eqref{eq:25}. 
We then have
\begin{equation}
\label{eq:37}
  \|\nabla w\|_{L^2(\Omega\setminus B)} \leq \|\nabla(\zeta\phi_0(P\cdot,\eta))\|_2 + \|\zeta\phi_0(P\cdot,\eta)\nabla\chi\|_2 +
  \|\chi\nabla\bar{\phi}_2\|_2 + \|\bar{\phi}_2\nabla\chi\|_2 \,,
\end{equation}
in which $P:\Omega\setminus\Omega_{\delta_0}\to\partial\Omega$ is the projection on the boundary
($P(x)=s$), where $\Omega_{\delta_0}$ is given by \eqref{eq:10}. 

For the first term on the right-hand-side of \eqref{eq:37} we have 
\begin{equation}
\label{eq:38}
   \|\nabla(\zeta\phi_0)\|_2^2 \leq  \frac{C}{\delta} 
   \|\phi_0(\cdot,\eta)\|_{L^2(\partial\Omega)}^2 + C\delta\|\nabla_s\phi_0(\cdot,\eta)\|_{L^2(\partial\Omega)}^2 \,,
\end{equation}
where $\nabla_s$ denotes the tangential derivative on $\partial\Omega$. By
\eqref{eq:29} and the smoothness of $\partial\Omega$ we have
\begin{multline}
\label{eq:39}
  \|\phi_0(\cdot,\eta)\|_{L^2(\partial\Omega)}^2 \leq C \epsilon^6 \int_{\partial\Omega}
  \frac{ds_\xi}{|\xi-\eta|^4} \leq\\ C \epsilon^6 \int_{\R^2}
  \frac{d\xi}{[|\xi|^2+d(\eta,\partial\Omega)^2]^2} \leq C \frac{\epsilon^6}{d(\eta,\partial\Omega)^2} \,.
\end{multline}
In a similar manner we obtain that
\begin{displaymath}
  \|\nabla_s\phi_0(\cdot,\eta)\|_{L^2(\partial\Omega)}^2 \leq C \frac{\epsilon^6}{d(\eta,\partial\Omega)^4} \,,
\end{displaymath}
which together with \eqref{eq:39} and \eqref{eq:38} yields
\begin{displaymath}
  \|\nabla(\zeta\phi_0)\|_2^2  \leq \frac{C}{\delta} \frac{\epsilon^6}{d(\eta,\partial\Omega)^2}
   + C\delta \frac{\epsilon^6}{d(\eta,\partial\Omega)^4} \,.
\end{displaymath}
Upon choosing
\begin{displaymath}
  \delta = \min\big(d(\eta,\partial\Omega),\delta_0\big) \,,
\end{displaymath}
we obtain
\begin{equation}
\label{eq:40}
 \|\nabla(\zeta\phi_0)\|_2   \leq C \frac{\epsilon^3}{d(\eta,\partial\Omega)^{3/2}} \,.
\end{equation}

For the second term on the right-hand-side of \eqref{eq:37} we have
\begin{equation}
\label{eq:41}
  \|\zeta\phi_o(P\cdot,\eta)\nabla\chi\|_2\leq \|\phi_0\|_{L^\infty(\partial\Omega)}\|\nabla\chi\|_2 \leq
  C\frac{\epsilon^{7/2}}{d(\eta,\partial\Omega)^2} \,.
\end{equation}
For the last two terms on the right-hand-side of
\eqref{eq:37} we easily obtain, using the fact that $\bar{\phi}\in C^{2,\alpha}(\bar(\Omega)$,
\begin{displaymath}
  \|\chi\nabla\bar{\phi}_2\|_2 + \|\bar{\phi}_2\nabla\chi\|_2  \leq
  C(\epsilon^{3/2}\|\nabla\bar{\phi}_2\|_\infty + \epsilon^{1/2}\|\bar{\phi}_2\|_\infty
  )\leq C\epsilon^{5/2} \,.
\end{displaymath}
Combining the above with \eqref{eq:41}, \eqref{eq:40},   yields
\begin{displaymath}
 \|\nabla w\|_{L^2(\Omega\setminus B)}\leq C(\Omega,\sigma) \Big( \epsilon^{5/2} + \frac{\epsilon^3}{d(\eta,\partial\Omega)^{3/2}}\Big) \,.
\end{displaymath}
By \eqref{eq:35} we have that
\begin{equation}
\label{eq:42}
   \|\nabla v_1\|_{L^2(B)}=\|\nabla \bar{\phi}_2\|_{L^2(B)} \leq C\varepsilon^{5/2} \,.
\end{equation}
Hence, since $w\in X_1(-\phi_0,-\bar{\phi}_2)$ we have that
\begin{equation}
    \label{eq:43}
 \|\nabla v_1\|_2\leq CI_1(W)+C\epsilon^{5/2}\leq C(\Omega) \Big( \epsilon^{5/2} + \frac{\epsilon^3}{d(\eta,\partial\Omega)^{3/2}}\Big) \,.
\end{equation}

Consider next the case $d(\eta,\partial\Omega)\leq2\epsilon$. Here we set
\begin{equation}
  \label{eq:44}
w_0(x,\eta) = -\phi_0(x,\eta) +\frac{[\bar{\phi}_2(\sigma,\eta)-\phi_0(\sigma,\eta)]-
  [\bar{\phi}_2(\tau,\eta)-\phi_0(\tau,\eta)]}{|\tau^*-\tau|}|\tau^*-x|\,,
\end{equation}
where $\sigma=P(\eta)$, $\tau\in\partial B$ is given by
\begin{displaymath}
  \tau=\eta+ \epsilon\frac{x-\eta}{|x-\eta|}\,,
\end{displaymath}
and $\tau^*\in\partial\Omega$ is chosen so that
\begin{displaymath}
 \frac{\tau^*-\eta}{|\tau^*-\eta|}= \frac{x-\eta}{|x-\eta|}\,.
\end{displaymath}
If $\Omega$ is convex, then $\tau^*$ is uniquely defined. Otherwise we choose
$\tau^*$ as the closest point to $\eta$ on $\partial\Omega$ in the above direction. We
then choose 
\begin{displaymath}
   v(x,\eta)= -\zeta(t) \phi_0(s,\eta)(1-\chi(|x-\eta|/\epsilon)) + \chi(|x-\eta|/\epsilon)w_0(x) \,,
\end{displaymath}
where $\zeta$ and $\chi$ are as above. Note that since
$\bar{\phi}_2(\sigma,\eta)-\phi_0(\sigma,\eta)$ is independent of $x$ we have
$v\in X_1(-\phi_0,-\bar{\phi}_2)$. Clearly,
\begin{multline}
\label{eq:45}
  \|\nabla v\|_{L^2(\Omega\setminus B)}\leq \|\nabla(\zeta\phi_0(P\cdot,\eta))\|_{L^2(\Omega\setminus B)} + \\ \|\zeta\phi_0(P\cdot,\eta)\nabla\chi\|_2 +
  \|\chi\nabla w_0\|_{L^2(\Omega\setminus B)} + \|w_0\nabla\chi\|_2\,.
\end{multline}
For the third term we have
\begin{multline*}
    \|\chi\nabla w_0\|_{L^2(\Omega\setminus B)}  \leq \|\chi\nabla \phi_0\|_{L^2(\Omega\setminus B)} +
    \Big\|\chi\frac{\bar{\phi}_2(\sigma,\eta)-\bar{\phi}_2(\tau,\eta)}{|\tau^*-\tau|}\Big\|_{L^2(\Omega\setminus B)}\\+
    \Big\|\chi\frac{\phi_0(\sigma,\eta)-\phi_0(\tau,\eta)}{|\tau^*-\tau|}\Big\|_{L^2(\Omega\setminus B)}
\end{multline*}
Estimating the boundary by its tangent plane at $s$ we obtain
\begin{multline*}
   \Big\|\chi\frac{\bar{\phi}_2(\sigma)-\bar{\phi}_2(\tau)}{|\tau^*-\tau|}\Big\|_{L^2(\Omega\setminus B)}^2 \leq
   C\epsilon^2 \Big\|\chi\frac{|\sigma-\tau|}{|\tau^*-\tau|}\Big\|_{L^2(\Omega\setminus B)}^2 \\ \leq
   C\epsilon^2\int_0^{\pi/4}\int_\epsilon^{h_\epsilon(\phi)}\frac{[\epsilon^2+(d+\epsilon)^2-2\epsilon(d+\epsilon)\cos\phi]\rho^2\sin\phi\cos^2\phi
     }{[(\epsilon+d)-\epsilon\cos\phi]^2}\,d\rho\,d\phi \leq C\epsilon^5\,,
\end{multline*}
where $h_\epsilon(\phi) = (d+\epsilon)\min(\sec \phi,2)$. Similarly, we obtain that
\begin{displaymath}
  \Big\|\chi\frac{\phi_0(\sigma,\eta)-\phi_0(\tau,\eta)}{|\tau^*-\tau|}\Big\|_{L^2(\Omega/B)}\leq C\epsilon^{3/2} \,.
\end{displaymath}
Hence,
\begin{equation}
\label{eq:46}
    \|\chi\nabla w_0\|_{L^2(\Omega\setminus B)}  \leq C\epsilon^{3/2} \,.
\end{equation}
In a similar manner we obtain that
\begin{displaymath}
  \|w_0\nabla\chi\|_2\leq C\epsilon^{3/2}
\end{displaymath}
Substituting the above, together with \eqref{eq:40}, \eqref{eq:41}, and
\eqref{eq:46} into \eqref{eq:45} then yields
\begin{displaymath}
   \|\nabla v\|_{L^2(\Omega\setminus B)}\leq C\epsilon^{3/2}\,,
\end{displaymath}
which together with \eqref{eq:42} and \eqref{eq:43}
yields \eqref{eq:34}.
\end{proof}
An immediate corollary follows
\begin{corollary}
There exists $C(\Omega,f,\lambda,\Lambda)>0$ such that
  \begin{equation}
\label{eq:47}
    \|\nabla\phi_1(\cdot,\eta)\|_2 \leq C\epsilon^{3/2} \,.
  \end{equation}
\end{corollary}

We continue by the following simple result
\begin{lemma}
  There exists $C(\Omega,f)>0$ such that
\begin{equation}
\label{eq:48}
  \|\phi_1(\cdot,\eta)\|_\infty \leq C\epsilon \,.
\end{equation}
\end{lemma}
\begin{proof}
By \eqref{eq:27} we have that
    \begin{equation}
\label{eq:49}
  \begin{cases}
    \LL \phi_1(\cdot,\eta) =0 & \text{in } \Omega\setminus B \\
    \phi_1(\cdot,\eta) = 0 & \text{ on } \partial\Omega \\
    \phi_1=C_1-\bar{\phi} & \text{in } B\,, \\
    \int_{\partial B} a\frac{\partial \phi_1}{\partial\nu} \,ds =0 \,. 
\end{cases}
  \end{equation}
It can be easily verified that $\phi_1$ is the minimizer of
\eqref{eq:15} in $X_1(0,-\bar{\phi})$
For fixed $C\in\R$ denote by $w_C$ the minimizer of \eqref{eq:15} in
$\HH_1(C,0,-\bar{\phi})$. Clearly,
\begin{displaymath}
   I_1(w_C) \geq \lambda\inf_{
     \begin{subarray}{c}
w\in H^1(\R^3\setminus B) \\
(w-w_C)|_{\partial B}= 0    
     \end{subarray}}
\|\nabla w\|_{L^2(\R^3\setminus B)}^2    \,.
\end{displaymath}
Let $\bar{C}=(\bar{\phi})_{\partial B}$, where $(\cdot)_U$ denotes the average on $U$. Since
\begin{displaymath}
   \inf_{     \begin{subarray}{c}
w\in H^1(\R^3\setminus B) \\
(w-w_C)|_{\partial B}= 0    
     \end{subarray}} \|
\nabla w\|_{L^2(\R^3\setminus B)}^2\geq4\pi|C-\bar{C}|^2\epsilon \,,
\end{displaymath}
and since $\phi_1=w_{C_1}$, we obtain from \eqref{eq:47} that
\begin{displaymath}
  4\pi\lambda|C_1-\bar{C}|^2\epsilon\leq   \|\nabla\phi_1(\cdot,\eta)\|_2^2 \leq C\epsilon^3\,.
\end{displaymath}
Hence,
\begin{displaymath}
  \|C_1-\bar{\phi}\|_{L^\infty(\partial B)} \leq  \|\bar{C}-\bar{\phi}\|_{L^\infty(\partial B)} +
  |C_1-\bar{C}| \leq C\epsilon \,.
\end{displaymath}
The lemma now follows from the maximum principle, both inside and
outside $B$.
\end{proof}

We next derive a local $L^2$ estimate for $\nabla\phi_1$.
\begin{lemma}
\label{lem:2.2}
   Let $\phi_1$ be given by \eqref{eq:28}. Then, for all
   $(z,\eta)\in\Omega\times\Omega_\epsilon$,
        \begin{equation}
   \label{eq:50}
\|\nabla\phi_1(\cdot,\eta)\|_{L^2\big(B(z,\epsilon)\cap\Omega\big)} \leq
C(\Omega,f,\lambda,\Lambda)\frac{\epsilon^{9/2}}{|z-\eta|^3}\,,
  \end{equation}
\end{lemma}
\begin{proof}
  We first take the gradient of \eqref{eq:19} with $N=1$ to obtain
  \begin{multline}
\label{eq:51}
     \nabla\phi_1(x,\eta)=\int_{\partial B}\nabla_x
  G(x,\xi)a(\xi)\frac{\partial \psi_1}{\partial\nu}(\xi,\eta)\,ds_\xi = \int_{\partial B}
  \nabla_xG\,a\Big(\frac{\partial\bar{\phi}}{\partial\nu}+\frac{\partial\phi_1}{\partial\nu}\Big)\,ds_\xi 
\end{multline}
for all $x\in\Omega\setminus B$. Consider first the case where
$|x-\eta|>3\epsilon$. By \eqref{eq:18} we have
\begin{displaymath}
   \int_{\partial B}a(\xi)\nabla_xG(x,\xi)\frac{\partial\bar{\phi}}{\partial\nu}(\xi)\,ds_\xi =
   \int_{B}\nabla_x\nabla_\xi G(x,\xi)\cdot a(\xi)\nabla\bar{\phi}(\xi)\,d\xi\,.
\end{displaymath}
In appendix A we show that
\begin{equation}
\label{eq:52}
  \|D^2G\|(x,\xi) \leq \frac{C(\Omega)}{|x-\xi|^3} \,,
\end{equation}
where $D^2G$ denotes the Hessian matrix of $G$.  Consequently, by \eqref{eq:18},
\begin{multline}
\label{eq:53}
\Big|\int_{\partial B}\nabla_xG(x,\xi)a\frac{\partial\bar{\phi}}{\partial\nu}\,ds_\xi\Big|=\Big|\int_{\partial
  B}[\nabla_xG(x,\xi)-\nabla_xG(x,\eta)]a \frac{\partial\bar{\phi}}{\partial\nu}\,ds_\xi\Big| \\\leq
C\frac{\epsilon^3}{|x-\eta|^3} \,.
\end{multline}
 Furthermore, in view of \eqref{eq:27}, \eqref{eq:28}, and \eqref{eq:18}, we have that
\begin{displaymath}
    \int_{\partial B}\nabla_xG(x,\xi)\frac{\partial\phi_1}{\partial\nu}\,ds_\xi = \int_{\partial
      B}[\nabla_xG(x,\xi)-\nabla_xG(x,\eta)]a\frac{\partial\phi_1}{\partial\nu}\,ds_\xi \,.
\end{displaymath}
Let $\chi$ be given by \eqref{eq:25} and set
$\chi_\epsilon=\chi(|x-\eta|/\epsilon)$. Suppose first that
$d(\eta,\partial\Omega)\geq2\epsilon$. Integration by parts then yields
\begin{multline*}
  \int_{\partial B}\nabla_xG(x,\xi)a\frac{\partial\phi_1}{\partial\nu}\,ds_\xi =
 \int_{\Omega\setminus B}\chi_\epsilon D^2_{x\xi}G\cdot a\nabla\phi_1\,d\xi \\ + \int_{\Omega\setminus B}[\nabla_xG(x,\xi)-\nabla_xG(x,\eta)]\nabla\chi_\epsilon\cdot a\nabla\phi_1\,d\xi\,.
\end{multline*}
Consequently, by \eqref{eq:52} and  \eqref{eq:47} we obtain that
\begin{multline}
\label{eq:54}
\Big|\int_{\partial B}\nabla_xG(x,\xi)a\frac{\partial\phi_1}{\partial\nu}\,ds_\xi\Big|
 \leq \|D^2G(x,\cdot)\|_{L^\infty( B(\eta,2\epsilon))}
 \|\nabla\phi_1(\cdot,\eta)\|_{L^1(\Omega\cap B(\eta,2\epsilon)\setminus B(\eta,\epsilon)}  \\ \leq
 C\frac{\epsilon^3}{|x-\eta|^3} \,.
\end{multline}

Next, consider the case $|x-\eta|\geq3\epsilon$, and $d(\eta,\partial\Omega)<2\epsilon$. Here we
write,
\begin{displaymath}
  \int_{\partial B}\nabla_xG(x,\xi)a\frac{\partial\phi_1}{\partial\nu}\,ds_\xi =
 \int_{\Omega\setminus B}\chi_\epsilon D^2_{x\xi}G\cdot a\nabla\phi_1\,d\xi+ \int_{\Omega\setminus B}\nabla_xG(x,\xi)\nabla\chi_\epsilon\cdot a\nabla\phi_1\,d\xi\,.
\end{displaymath}
The first integral on the right-hand-side can be estimated in
precisely the same manner as in \eqref{eq:54} to obtain that
\begin{displaymath}
  \Big|\int_{\Omega\setminus B}\chi_\epsilon D^2_{x\xi}G\cdot a\nabla\phi_1\,d\xi\Big| \leq C\frac{\epsilon^3}{|x-\eta|^3} \,.
\end{displaymath}
To estimate the second integral we first observe that since
$\nabla_xG(x,\cdot)\equiv0$ on $\partial\Omega$, and hence
\begin{equation}
\label{eq:55}
  \|\nabla_xG(x,\cdot)\cap\Omega\|_{L^\infty(B(\eta,2\epsilon)\cap\Omega}\leq
  4\epsilon\|D^2G(x,\cdot)\|_{L^\infty(B(\eta,2\epsilon)\cap\Omega} \,.
\end{equation}
Hence, as in \eqref{eq:53} we have
\begin{multline*}
  \Big|\int_{\Omega\setminus B}\nabla_xG(x,\xi)\nabla\chi_\epsilon\cdot a\nabla\phi_1\,d\xi\Big| \leq \\ C\|D^2G(x,\cdot)\|_{L^\infty( B(\eta,2\epsilon))}
 \|\nabla\phi_1(\cdot,\eta)\|_{L^1(\Omega\cap B(\eta,2\epsilon)\setminus B(\eta,\epsilon)} \leq
 C\frac{\epsilon^3}{|x-\eta|^3} \,.
\end{multline*}
The above, in conjunction with \eqref{eq:53} and \eqref{eq:54}, leads to
\begin{equation}
   \label{eq:56}
|\nabla\phi_1(x,\eta)|\leq  C\frac{\epsilon^3}{|x-\eta|^3} \,,
\end{equation}
from which (\ref{eq:50}) easily follows for the case $|z-\eta|\geq4\epsilon$. When $|z-\eta|\leq4\epsilon$ it
immediately follows from \eqref{eq:47}. 
\end{proof}

\section{Two inclusions}
\label{sec:twoinclusions}
We now proceed to consider a two-particle problem. Let
$\psi_2(\cdot,\eta_1,\eta_2):\Omega\to\R$ denote, for every
$(\eta_1,\eta_2)\in\Omega_\epsilon\times\Omega_\epsilon$, the unique (weak) solution of 
\begin{equation}
\label{eq:57}
  \begin{cases}
    \LL \psi_2(x,\eta_1,\eta_2)) =0 &  x\in\Omega\setminus(B_1\cup B_2) \\
    \psi_2(x,\eta_1,\eta_2) = f & x\in\partial\Omega \,, \\
    \psi_2=C_i & \text{in } B_i \; i=1,2 \,.\,, \\
    \int_{\partial B_i} a\frac{\partial\psi_2}{\partial\nu} \,ds =0& i=1,2 \,, 
\end{cases}
\end{equation}
where $B_i=B(\eta_i,\epsilon)$. 
We then set
\begin{equation}
\label{eq:58}
  v_2(\cdot,\eta_1,\eta_2) = \psi_2(\cdot,\eta_1,\eta_2) -\bar{\phi}- \phi_1(\cdot,\eta_1) -
  \phi_1(\cdot,\eta_2) \,.
\end{equation}
For convenience of notation we set
\begin{displaymath}
  \|\cdot\|_{2,o}=\|\cdot\|_{L^2(\Omega\setminus\Union_{n=1}^2B_i)}\,.
\end{displaymath}

We begin with the following global estimate
\begin{lemma}
  Let $\delta_i$ be given by \eqref{eq:22} and $v_2$ be defined by
  \eqref{eq:58}. Then
  \begin{equation}
    \label{eq:59}
\|\nabla v_2(\cdot,\eta_1,\eta_2)\|_2 \leq C \frac{\epsilon^{9/2}}{|\eta_2-\eta_1|^3}[1+(|\ln \delta_1|+|\ln
\delta_2|){\mathbf 1}_{B(\eta_1,4\epsilon)}(\eta_2)]^{1/2}\,.
  \end{equation}
\end{lemma}
\begin{proof}
  It can be easily verified that
\begin{subequations}
  \label{eq:60}
\begin{empheq}[left={\empheqlbrace}]{alignat=2}    
 &   \LL v_2(\cdot,\eta_1,\eta_2)) =0 & \quad & \text{in }\Omega\,, \\
  &  v_2(\cdot,\eta_1,\eta_2) = 0 &\quad & \text{on }\partial\Omega \,, \\
  &  v_2=C_i-\phi_1(\cdot,\eta_{3-i}) & \quad &\text{in } B_i \; i=1,2 \,, \\
   & \int_{\partial B_i} a\frac{\partial v_2}{\partial\nu} \,ds =0&\quad & i=1,2 \,.    
  \end{empheq}
\end{subequations}
We note further that $v_2$ is the minimizer in
 of $I_2$ in $X_2(0,-\{\phi_1(\cdot,\eta_{3-i})\}_{i=1}^2)$, respectively given
by \eqref{eq:15} and \eqref{eq:16}.

Consider first the case where $|\eta_2-\eta_1|>4\epsilon$ and
$\delta_1=\delta_2=1$.  Let
$\chi_\epsilon^i(x)=\chi(|x-\eta_i|/\epsilon)$ ($i=1,2$), where $\chi$ is defined by
\eqref{eq:25}. Then, set
\begin{equation}
\label{eq:61}
  w=-\chi_1\big(\phi_1(\cdot,\eta_2)-\phi_1(\eta_1,\eta_2)\big)  - \chi_2\big(\phi_1(\cdot,\eta_1)-\phi_1(\eta_2,\eta_1)\big) \,.
\end{equation}
Clearly, $w\in X_2(0,-\{\phi_1(\cdot,\eta_{3-i})\}_{i=1}^2)$. Consequently,
\begin{multline}
  \label{eq:62}
\|a^{1/2}\nabla v_2\|_{2,o}\leq \|a^{1/2}\nabla w\|_{2,o} \leq C\sum_{i=1}^2\|\chi_i\nabla\phi_1(\cdot,\eta_{3-i}))\|_{2,o} \\+
\big\|\big(\phi_1(\cdot,\eta_{3-i})-\phi_1(\eta_i,\eta_{3-i})\big)\nabla\chi_i\big\|_{2,o} \,.
\end{multline}
By (\ref{eq:50}) we have that
\begin{displaymath}
  \|\chi_i\nabla\phi_1(\cdot,\eta_{3-i}))\|_2  \leq C \frac{\epsilon^{9/2}}{|\eta_2-\eta_1|^3}\,.
\end{displaymath}
Furthermore, by \eqref{eq:56} we have that
\begin{multline*}
  \Sigma_{i=1}^2\big\|\big(\phi_1(\cdot,\eta_{3-i})-\phi_1(\eta_i,\eta_{3-i})\big)\nabla\chi_i\big\|_2\leq  \\  C\epsilon^{3/2}\big(( \|\nabla\phi_1(\cdot,\eta_2)\|_{L^\infty(B(\eta_1,2\epsilon))}  
+  \|\nabla\phi_1(\cdot,\eta_1)\|_{L^\infty(B(\eta_2,2\epsilon))}\big) \leq C \frac{\epsilon^{9/2}}{|\eta_2-\eta_1|^3}\,.
\end{multline*}
Hence, by \eqref{eq:62} we obtain that
\begin{equation}
  \label{eq:63}
\|\nabla v_2(\cdot,\eta_1,\eta_2)\|_{2,o}\leq C \frac{\epsilon^{9/2}}{|\eta_2-\eta_1|^3}\,.
\end{equation}

Next we consider the case $|\eta_2-\eta_1|<4\epsilon$ and $\delta_1=\delta_2=1$. Let 
\begin{equation}
\label{eq:64}
  U_m=\big(B(\eta_1,m\epsilon)\cup B(\eta_2,m\epsilon) \big)\cap\Omega
\end{equation}
and define the cutoff function $\zeta_2\in
C^\infty_0(\Omega,[0,1])$
\begin{equation}
\label{eq:65}
  \zeta_2(x)=
  \begin{cases}
    1 & x\in U_1 \\
    0 & x\in\Omega\setminus U_2 \,.
  \end{cases} \quad |\nabla\zeta|\leq \frac{C}{\epsilon}\,.
\end{equation}
Then we set
\begin{equation}
\label{eq:66}
  w= -\zeta_2[(\bar{\phi}+\phi_1(\cdot,\eta_1) +\phi_1(\cdot,\eta_2))-
  \bar{\phi}(\eta_1)] \,.
\end{equation}

It can be easily verified from \eqref{eq:49} that
$w\in X_2(0,-\{\phi_1(\cdot,\eta_{3-i})\}_{i=1}^2)$. Consequently,
\begin{multline*}
  \|\nabla v_2\|_{2,o}\leq  C\big(\|\zeta_2\nabla\bar{\phi}\|_2 +
  \|(\bar{\phi}-\bar{\phi}(\eta_1))\nabla\zeta\|_2  +
  \|\nabla\phi_1(\cdot,\eta_1)\|_2+ \|\nabla\phi_1(\cdot,\eta_2)\|_2 + \\
  \|[\phi_1(\cdot,\eta_1)-(\tilde{\phi}_1(\cdot,\eta_1))_{B(\bar{\eta},4\epsilon)}]\nabla\zeta\|_2+\|[\phi_1(\cdot,\eta_2)-(\tilde{\phi}_1(\cdot,\eta_2))_{B(\bar{\eta},4\epsilon)}]\nabla\zeta\|_2 \big)\,. 
\end{multline*}
Since $\bar{\phi}\in C^{2,\alpha}(\bar{\Omega})$ we have
\begin{displaymath}
  \|\zeta\nabla\bar{\phi}\|_2 +
  \|(\bar{\phi}-\bar{\phi}(\eta_1))\nabla\zeta\|_2   \leq C\epsilon^{3/2} \,.
\end{displaymath}
By \eqref{eq:47} we have that
\begin{displaymath}
  \|\nabla\phi_1(\cdot,\eta_1)\|_2+ \|\nabla\phi_1(\cdot,\eta_2)\|_2 \leq  C\epsilon^{3/2} \,.
\end{displaymath}
Finally, by \eqref{eq:48}, we obtain that
\begin{displaymath}
 \|\phi_1(\cdot,\eta_1)\nabla\zeta\|_2+\|\phi_1(\cdot,\eta_2)\nabla\zeta\|_2 \leq  C\epsilon^{3/2}\,.
\end{displaymath}
Hence
\begin{equation}
  \label{eq:67}
 \|\nabla v_2(\cdot,\eta_1,\eta_2)\|_{2,o}\leq C\epsilon^{3/2}\,.
\end{equation}

Consider next the case $\min(\delta_1,\delta_2)<1$ and
$|\eta_2-\eta_1|\geq4\epsilon$. Let $z_i\in B(\eta_{3-i},2\epsilon)$ satisfy
\begin{displaymath}
  |\phi_1(z_i,\eta_i)| = \min_{x\in B(\eta_{3-i},2\epsilon)}|\phi_1(x,\eta_i)|\,,
\end{displaymath}
and $K_i=\phi_1(z_i,\eta_i)$ for $i=1,2$. Then, we set 
\begin{displaymath}
  w=-\chi_1(\phi_1(\cdot,\eta_2)-K_2) -\chi_2(\phi_1(\cdot,\eta_1)-K_1) \,.
\end{displaymath}
Since $\phi_1(\cdot,\eta)\in H^1_0(\Omega)$ we have $w\in
X_2(0,-\{\phi_1(\cdot,\eta_{3-i})\}_{i=1}^2)$. As in \eqref{eq:62} we then
obtain that
\begin{multline}
\label{eq:68}
  \|\nabla v_2(\cdot,\eta_1,\eta_2)\|_{2,o}\leq C\sum_{i=1}^2\big[\|\chi_{3-i}\nabla\phi_1(\cdot,\eta_i))\|_2 +\\
\big\|(\phi_1(\cdot,\eta_i)-K_i)\nabla\chi_{3-i}\big\|_2\big]\leq C \frac{\epsilon^{9/2}}{|\eta_2-\eta_1|^3}\,. 
\end{multline}

Finally, we consider the case $\min(\delta_1,\delta_2)<1$
and $|\eta_2-\eta_1|<4\epsilon$.  Let $C_2(\eta_1,\eta_2)$ be given by \eqref{eq:20}
and let $\kappa_2(\cdot,\eta_1,\eta_2)\in\W(\eta_1,\eta_2)$ denote its associated
minimizer. Let 
$\zeta_2\in C^\infty(\Omega,[0,1])$ satisfy
\begin{equation}
\label{eq:69}
  \zeta_2(x,\eta_1,\eta_2)=
  \begin{cases}
    1 & x\in U_1 \\
    0 & x\in\Omega\setminus U_2
  \end{cases}
\quad |\nabla\zeta_2|\leq \frac{C}{\epsilon}
\end{equation}
Set then
\begin{displaymath}
  w= -\zeta_2\kappa_2(\bar{\phi}+\phi_1(\cdot,\eta_1) +\phi_1(\cdot,\eta_2) -\tilde{K}_2)\,,
\end{displaymath}
where
\begin{displaymath}
  \tilde{K}_2=(\bar{\phi})_{U_2}\,,
\end{displaymath} 
(recall that $(\cdot)_U$ denotes the average over $U$).  It can be easily
verified that $w\in X_2(0,-\{\phi_1(\cdot,\eta_{3-i})\}_{i=1}^2)$. Furthermore, by
\eqref{eq:47}, \eqref{eq:48}, and \eqref{eq:23} we have
\begin{multline*}
  \|\nabla w\|_2\leq \|\bar{\phi}+\phi_1(\cdot,\eta_1) +\phi_1(\cdot,\eta_2)
  -\tilde{K}_2\|_{L^\infty(U_2)}(\|\nabla \kappa_2\|_2
  +\|\nabla\zeta_2\|_2) + \|\nabla\bar{\phi}\|_{L^2(U_2\cap\Omega)} \\+
  \|\nabla\phi_1(\cdot,\eta_1)\|_{L^2(U_2)}  + \|\nabla\phi_1(\cdot,\eta_2)\|_{L^2(U_2)} \leq
  C\epsilon^{3/2}[1+|\ln \delta_1|+|\ln \delta_2|]^{1/2}\,.
\end{multline*}
By the above, \eqref{eq:63},  \eqref{eq:67}, and \eqref{eq:68} we have, thus, established that
\begin{displaymath}
   \|\nabla v_2(\cdot,\eta_1,\eta_2)\|_{2,o}\leq C \frac{\epsilon^{9/2}}{|\eta_2-\eta_1|^3}\big[1+(|\ln \delta_1|+|\ln
\delta_2|){\mathbf 1}_{B(\eta_1,4\epsilon)}(\eta_2)\big]^{1/2}\,.
\end{displaymath}
To complete the proof, we use (\ref{eq:60}c) and \eqref{eq:50} to obtain that
\begin{displaymath}
  \|\nabla v_2(\cdot,\eta_1,\eta_2)\|_{L^2(U_1)} \leq C \frac{\epsilon^{9/2}}{|\eta_2-\eta_1|^3}\,.
\end{displaymath}
The lemma is proved.
\end{proof}

We next establish the following $L^\infty$ estimate, analogously to \eqref{eq:48},
\begin{lemma}
   There exists $C(\Omega,f,\lambda,\Lambda)>0$ such that
\begin{equation}
\label{eq:70}
  \|v_2(\cdot,\eta_1,\eta_2)\|_\infty \leq C\frac{\epsilon^4}{|\eta_1-\eta_2|^3}[1+(|\ln \delta_1|+|\ln \delta_2|){\mathbf 1}_{B(\eta_1,4\epsilon)}(\eta_2)]^{1/2} \,.
\end{equation}
\end{lemma}
\begin{proof}
  Recall that that $v_2$ is the minimizer of $I_2(w)$ given by
  \eqref{eq:15} over all $w\in X_2(0,-\{\phi_1(\cdot,\eta_{3-i})\}_{i=1}^2)$
  given by \eqref{eq:16}. For fixed ${\mathbf K}=(K_1,K_2)\in\R^2$ denote by
  $w_{\mathbf K}$ the minimizer of $I_2(w)$ in
  $\HH_2(K_1,K_2,0,-\{\phi_1(\cdot,\eta_{3-i})\}_{i=1}^2))$, given by (\ref{eq:16}b).  Let, for
  $i=1,2$,
  \begin{displaymath}
    \tilde{w}_{\mathbf K}^i =
   \begin{cases}
      w_{\mathbf K} & \text{in } \Omega\setminus U_1\\
     K_{3-i} - \phi_1(\cdot,\eta_i) & \text{in } B_{3-i} \\
     0 & x\in\R^3\setminus\Omega \,.
   \end{cases}
  \end{displaymath}
Clearly, $ \tilde{w}_{\mathbf K}^i \in H^1(\R^3\setminus B_i)$ and hence
\begin{equation}
\label{eq:71}
   I_2(w_{\mathbf K}) +
   \|a^{1/2}\nabla\phi_1(\cdot,\eta_i)\|_{L^2(B_{3-i})}^2=\|a^{1/2}\nabla\tilde{w}_{\mathbf K}^i
   \|_{L^2(\R^3\setminus B_i)}^2\geq \lambda\inf_{w\in\Vg_i} \|\nabla
   w\|_{L^2(\R^3\setminus B_i)}^2\,,
\end{equation}
where
\begin{displaymath}
  \Vg_i(\mathbf K)=\{ w\in H^1(\R^3\setminus B_i) \,|\, (w-w_{\mathbf K})|_{B_i}=0\,\}
  \,. 
\end{displaymath}

Next, we set  $\bar{C}_i=(\phi_1(\cdot,\eta_{3-i}))_{\partial B_i}$. By
\eqref{eq:50} and \eqref{eq:71} we have
that 
\begin{displaymath}
  I_2(w_{\mathbf K})\geq 4\pi\lambda|K_i-\bar{C}_i|^2\epsilon-  C\frac{\epsilon^9}{|\eta_1-\eta_2|^6} \,.
\end{displaymath}
Let ${\mathbf C}=(C_1,C_2)$. Since $v_2=w_{\mathbf C}$, we obtain from \eqref{eq:59} that
\begin{multline*}
  4\pi\lambda\max_{i\in\{1,2\}}|C_i-\bar{C}_i|^2\epsilon -
  C\frac{\epsilon^9}{|\eta_1-\eta_2|^6} \leq   \|a^{1/2}\nabla v_2(\cdot,\eta_1,\eta_2)\|_{2,o}^2 \leq\\
  C\frac{\epsilon^9}{|\eta_1-\eta_2|^6}[1+(|\ln \delta_1|+|\ln
\delta_2|){\mathbf 1}_{B(\eta_1,4\epsilon)}(\eta_2)]  \,.
\end{multline*}
When $|\eta_1-\eta_2|>3\epsilon$ we have by \eqref{eq:56}, for i=1,2,
\begin{displaymath}
  \|\bar{C}_i-\phi_1(\cdot,\eta_{3-i})\|_{L^\infty(\partial B_i)} \leq  C\frac{\epsilon^4}{|\eta_1-\eta_2|^3}\,,
\end{displaymath}
whereas for $|\eta_1-\eta_2|\leq3\epsilon$ we have by \eqref{eq:48} that
\begin{displaymath}
  \|\bar{C}_i-\phi_1(\cdot,\eta_{3-i})\|_{L^\infty(\partial B_i)} \leq C\epsilon \,.
\end{displaymath}
Consequently,
\begin{multline*}
  \|C_i-\phi_1(\cdot,\eta_{3-i})\|_{L^\infty(\partial B_i)} \leq  \|\bar{C}_i-\phi_1(\cdot,\eta_{3-i})\|_{L^\infty(\partial B_i)} +
  |C_i-\bar{C}_i| \leq\\ C\frac{\epsilon^4}{|\eta_1-\eta_2|^3}\big[1+(|\ln \delta_1|+|\ln
\delta_2|){\mathbf 1}_{B(\eta_1,4\epsilon)}(\eta_2)\big]^{1/2} \,.
\end{multline*}
The lemma now follows from the maximum principle. 
\end{proof}

Finally, we establish a local $L_2$ estimate, as in \eqref{eq:50} for $\nabla v_2$.
\begin{lemma}
  Let $v_2$ be given by \eqref{eq:58}. Then, for all $z\in\Omega$ and
$(\eta_1,\eta_2)\in\Omega_\epsilon\times\Omega_\epsilon$ such that $|\eta_1-\eta_2|\geq2\epsilon$ we have 
  \begin{multline}
\label{eq:72}
    \|\nabla v_2(\cdot,\eta_1,\eta_2)\|_{L^2(B(z,\epsilon)\cap\Omega)} \leq C(\Omega,f,\lambda,\Lambda)
       \frac{\epsilon^{15/2}}{|\eta_1-\eta_2|^3}\times \\\Big[ \frac{1}{|z-\eta_1|^3}
    +  \frac{1}{|z-\eta_2|^3}\Big]\big[1+(|\ln \delta_1|+|\ln
\delta_2|){\mathbf 1}_{B(\eta_1,4\epsilon)}(\eta_2)\big]^{1/2}\,.
     \end{multline}
\end{lemma}
\begin{proof}
Since by \eqref{eq:19} for $N=2$ we have
  \begin{displaymath}
    \psi_2(x,\eta_1,\eta_2)=\bar{\phi}(x) + \sum_{n=1}^2\int_{\partial B_n} 
G(x,\xi)a(\xi)\frac{\partial\psi_2}{\partial\nu}(\xi,\eta_1,\eta_2) \,ds_\xi \,,
  \end{displaymath}
it can be easily verified from \eqref{eq:58}, \eqref{eq:28}, and
\eqref{eq:51} that for any $x\in\Omega\setminus U_1$
  \begin{displaymath}
    \nabla v_2(x,\eta_1,\eta_2)=\sum_{n=1}^2\int_{\partial B_n} 
\nabla_xG(x,\xi)a(\xi)\Big[\frac{\partial v_2}{\partial\nu}(\xi,\eta_1,\eta_2)-
\frac{\partial\phi_1}{\partial\nu}(\xi,\eta_{3-n})]\,ds_\xi\,.  
  \end{displaymath}
  Consider first the case $d(x,U_2)>\epsilon$.  When $d(U_1,\partial\Omega)>\epsilon$ we
  observe, in view of \eqref{eq:60}, that for $n=1,2$
  \begin{multline*}
    \int_{\partial B_n} 
\nabla_xG(x,\xi)a(\xi) \frac{\partial v_2}{\partial\nu}(\xi,\eta_1,\eta_2)\,ds_\xi=\\ \int_{\partial B_n} 
[\nabla_xG(x,\xi)-\nabla_xG(x,\eta)]a(\xi)\frac{\partial v_2}{\partial\nu}(\xi,\eta_1,\eta_2)\,ds_\xi\,.  
  \end{multline*} 
Integration by parts then yields
\begin{multline*}
  \int_{\partial U_1} 
[\nabla_xG(x,\xi)-\nabla_xG(x,\eta)]a(\xi)\frac{\partial v_2}{\partial\nu}(\xi,\eta_1,\eta_2)ds_\xi =\\
\int_{\Omega\setminus U_1}\big\{\zeta_2D^2_{x\xi}G+[\nabla_xG(x,\xi)-\nabla_xG(x,\eta)]\nabla\zeta_2\big\}\cdot
a\nabla v_2\,d\xi\,,
\end{multline*}
where $\zeta_2$ is given by \eqref{eq:69}. By \eqref{eq:52} and
\eqref{eq:59} we have that
\begin{multline}
\label{eq:73}
  \int_{\Omega\setminus U_1}|\zeta_2D^2_{x\xi}G\cdot a\nabla v_2|\,d\xi\leq
  C\frac{\epsilon^{3/2}}{d(x,U_1)^3}\|\nabla v_2\|_{2,o} \leq \\ \frac{C\epsilon^6}{|\eta_1-\eta_2|^3}\Big[ \frac{1}{|x-\eta_1|^3}
    +  \frac{1}{|x-\eta_2|^3}\Big]\big[1+(|\ln \delta_1|+|\ln
\delta_2|){\mathbf 1}_{B(\eta_1,4\epsilon)}(\eta_2)\big]^{1/2}\,,
\end{multline}
Similarly,
\begin{multline}
\label{eq:74}
  \int_{\Omega\setminus U_1}\big|\nabla_xG(x,\xi)-\nabla_xG(x,\eta)\big|\,|\nabla\zeta_2|a
|\nabla v_2|\,d\xi\leq
\epsilon\|D^2_{x\xi}G\|_{L^\infty(U_2)}\frac{C}{\epsilon}\|\nabla v_2\|_{L^1(U_2\setminus U_1)}\\ \leq \frac{C\epsilon^6}{|\eta_1-\eta_2|^3}\Big[ \frac{1}{|x-\eta_1|^3}
    +  \frac{1}{|x-\eta_2|^3}\Big] \big[1+(|\ln \delta_1|+|\ln
\delta_2|){\mathbf 1}_{B(\eta_1,4\epsilon)}(\eta_2)\big]^{1/2}\,. 
\end{multline}

When $d(U_1,\partial\Omega)\leq\epsilon$ we have
\begin{multline}
\label{eq:75}
  \int_{\partial U_1} \nabla_xG(x,\xi)a(\xi)\frac{\partial v_2}{\partial\nu}(\xi,\eta_1,\eta_2)ds_\xi =\\
\int_{\Omega\setminus U_1}\big\{\zeta_2D^2_{x\xi}G+ \nabla_xG(x,\xi)\nabla\zeta_2\big\}\cdot
a\nabla v_2\,d\xi\,.
\end{multline}
For the second term in the curly braces we have, in view of
\eqref{eq:55}, that
\begin{multline*}
  \int_{\Omega\setminus U_1}\nabla_xG(x,\xi)\nabla\zeta_2\cdot
a\nabla v_2\,d\xi\leq \frac{C}{d(x,U_1)^3} \|\nabla v_2\|_{L^1(\Omega\cap U_2\setminus U_1}\leq \\
\frac{C\epsilon^6}{|\eta_1-\eta_2|^3}\Big[ \frac{1}{|x-\eta_1|^3}
    +  \frac{1}{|x-\eta_2|^3}\Big]|\big[1+(|\ln \delta_1|+|\ln
\delta_2|){\mathbf 1}_{B(\eta_1,4\epsilon)}(\eta_2)\big]^{1/2}\,. 
\end{multline*}
As \eqref{eq:73} still holds when $d(U_1,\partial\Omega)\leq\epsilon$, we may use
the above, together with \eqref{eq:73} and \eqref{eq:75} to obtain
\begin{multline*}
\Big|\int_{\partial U_1} \nabla_xG(x,\xi)a(\xi)\frac{\partial
  v_2}{\partial\nu}(\xi,\eta_1,\eta_2)ds_\xi|\leq\\ \frac{C\epsilon^6}{|\eta_1-\eta_2|^3}\Big[ \frac{1}{|x-\eta_1|^3}
    +  \frac{1}{|x-\eta_2|^3}\Big]\big[1+(|\ln \delta_1|+|\ln\delta_2|){\mathbf
      1}_{B(\eta_1,4\epsilon)}(\eta_2)\big]^{1/2}\,.
\end{multline*}
In conjunction with \eqref{eq:73} and \eqref{eq:74} the above
inequality yields for all $U_1\subset\Omega$
\begin{multline}
\label{eq:76}
  \Big|\int_{\partial U_1} 
\nabla_xG(x,\xi)a(\xi) \frac{\partial v_2}{\partial\nu}(\xi,\eta_1,\eta_2)\,ds_\xi\Big| \leq\\ \frac{C\epsilon^6}{|\eta_1-\eta_2|^3}\Big[ \frac{1}{|x-\eta_1|^3}
    +  \frac{1}{|x-\eta_2|^3}\Big]|[1+(|\ln \delta_1|+|\ln\delta_2|){\mathbf
      1}_{B(\eta_1,4\epsilon)}(\eta_2)]^{1/2}\,.
\end{multline}

Finally, for $n=1,2$, we have whenever $\eta_n\in\Omega_\epsilon$
\begin{displaymath}
 -\int_{\partial B_n}
  [\nabla_xG(x,\xi)-\nabla_xG(x,\eta)]a(\xi)\frac{\partial\phi_1}{\partial\nu}(\xi,\eta_{3-n})\,ds_\xi=
  \int_{B_n}D^2_{x\xi}G\cdot a\nabla \phi_1(\xi,\eta_{3-n})\,d\xi\,.
\end{displaymath}
By \eqref{eq:52} and \eqref{eq:50} we then obtain  for $n=1,2$ and $d(x,U_2)>\epsilon$
\begin{displaymath}
  \int_{B_n}|D^2_{x\xi}G\cdot a\nabla \phi_1(\xi,\eta_{3-n})|\,d\xi \leq
  C\frac{\epsilon^6}{|\eta_1-\eta_2|^3|x-\eta_n|^3} \,.
\end{displaymath}
Combining the above with \eqref{eq:76} yields
\begin{equation}
\label{eq:77}
   |\nabla v_2(x,\eta_1,\eta_2)| \leq \frac{C\epsilon^6}{|\eta_1-\eta_2|^3}\Big[ \frac{1}{|x-\eta_1|^3}
    +  \frac{1}{|x-\eta_2|^3}\Big]|[1+(|\ln \delta_1|+|\ln\delta_2|){\mathbf
      1}_{B(\eta_1,4\epsilon)}(\eta_2)]^{1/2}\,,
 \end{equation}
from which \eqref{eq:72} readily follows for $d(x,U_2)>\epsilon$.  If
$d(x,U_2)\leq\epsilon$ the lemma follows immediately from \eqref{eq:59}. 
\end{proof}

\section{Error estimates}
\label{sec:error-estimates}
\begin{equation}
 \label{eq:78}
\phi(x,\eta_1,\ldots,\eta_N)=\bar{\phi}(x) + \sum_{i=1}^N \Big[ \phi_1(x,\eta_i) +
\frac{1}{2}\sum_{
  \begin{subarray}{2}
    j=1 \\
    j\neq i
  \end{subarray}}^N v_2(x,\eta_i,\eta_j) \Big] + u \,,
\end{equation}
in which $\phi_1$ is defined by \eqref{eq:28} and $v_2$ by
\eqref{eq:58}.  By \eqref{eq:1}, \eqref{eq:49}, and \eqref{eq:60}
$(u,\{C_n\}_{n=1}^N)$ is the solution of
\begin{subequations}
 \label{eq:79} 
\begin{empheq}[left={\empheqlbrace}]{alignat=2}  
  &  \LL u  = 0 & \text{in } \Omega\setminus\Union_{n=1}^N B_n \,, \\
   & u= 0 & \text{on } \partial\Omega \,, \\
  &  u=C_n-\frac{1}{2}\sum_{
    \begin{subarray}{c}\strut
      k=1 \\
      k\neq n
    \end{subarray}}^N
\sum_{
    \begin{subarray}{c}\strut
      m=1 \\
      m\neq k,n
    \end{subarray}}^N v_2(\cdot,\eta_k,\eta_m) & \quad\text{ in } B_n\,,\;1\leq n\leq N \,, \\
  &  \int_{\partial B_n} a\frac{\partial u}{\partial\nu} \,ds =0 \,. &
  \end{empheq}
\end{subequations}
 Clearly, the restriction of 
$u$ to $\Omega\setminus\Union_{n=1}^NB_n$ is the minimizer of \eqref{eq:15} in 
\begin{displaymath}
  Y_N\overset{def}{=}X_N\bigg(0,\bigg\{C_n- \frac{1}{2}\sum_{
    \begin{subarray}{c}\strut
      k=1 \\
      k\neq n
    \end{subarray}}^N
\sum_{
    \begin{subarray}{c}\strut
      m=1 \\
      m\neq k,n
    \end{subarray}}^N v_2(\cdot,\eta_k,\eta_m) |_{\partial B_n}\bigg\}_{n=1}^N \bigg)\,.
\end{displaymath}
For convenience, we define, as in \S\,4 the norm
\begin{displaymath}
  \|\cdot\|_{2,o}= \|\cdot\|_{L^2\big(\Omega\setminus\Union_{n=1}^NB_n\big)} \,.
\end{displaymath}

We now set
\begin{displaymath}
  \{1,\ldots,N\}=\Union_{k=1}^M \J_k\,
\end{displaymath}
where the $\J_k$'s are selected so that
\begin{displaymath}
  \max_{m\in\J_k} \min_{l\in\J_k} |\eta_m-\eta_l|\leq4\epsilon \,,
\end{displaymath}
and
\begin{displaymath}
  \min_{(m,n)\in\J_k\times\J_j}|\eta_m-\eta_n|>4\epsilon
  \quad k\neq j\,.
\end{displaymath}
It can be easily verified that the above selection exists and is
unique. For convenience of notation we also set for $1\leq j\leq M$
\begin{displaymath}
  \J_j^c = \{1,\ldots,N\}\setminus\J_j
\end{displaymath}
In a similar manner to \eqref{eq:64} we then define
\begin{equation}
\label{eq:80}
  \Ug_j^n=\Big(\Union_{m\in\J_j} B(\eta_m,n\epsilon)\Big)\cap\Omega \,.
\end{equation}
Thus,
\begin{displaymath}
  \Ug_j^0 = \{\eta_m\}_{m\in\J_j}\,.
\end{displaymath}
We can now begin our attempt to construct a test function
$\tilde{u}\in Y_N$ in the form
\begin{equation}
\label{eq:81}
  \tilde{u}=
  \begin{cases}
    \sum_{j=1}^M u_j & \text{in }\Omega\setminus\Union_{n=1}^NB_n \\
    u +C_n &  \text{in } B_n \;\forall1\leq n\leq N \,,
  \end{cases}
\end{equation}
where $u_j$ is supported on $\Ug_j^2$. 

We further set 
\begin{displaymath}
  \Kg_n = \Union_{j:\,|\J_j|=n} \J_j \,,
\end{displaymath}
and then,
\begin{displaymath}
  \Ig_m = \Union_{n\geq m} \Kg_n\,.
\end{displaymath}
Let $\tilde{u}$ be given by \eqref{eq:81}. For 
$\tilde{u}\in Y_N$ we have $\|a^{1/2}\nabla u\|_{2,o}\leq
\|a^{1/2}\nabla\tilde{u}\|_{2,o}$. 
We shall construct $u_j$ so that 
\begin{equation}
\label{eq:82}
  u_j=u+C_k \text{ in } B_k \; \forall k\in\J_j \,,
\end{equation}
and hence  $\|a^{1/2}\nabla u\|_2\leq\|a^{1/2}\nabla\tilde{u}\|_2$. Consequently,
\begin{displaymath}
  \Eb_f\Big( \|\nabla u\|_2^2\Big) \leq C\Eb_f\Big( \|\nabla \tilde{u}\|_2^2\Big)\,.
\end{displaymath}
Note further that by the mutually disjoint support of the $u_j$s we
have 
\begin{displaymath}
  \Eb_f\Big( \|\nabla u\|_2^2\Big) \leq C\sum_{j=1}^M   \Eb_f\Big( \|\nabla
  u_j\|_2^2\Big)\,.
\end{displaymath}

The estimate of $ \Eb_f\Big( \|\nabla
  u_j\|_2^2\Big)$ is split in the following into three different
  cases: $j\in\Kg_1$, $j\in\Kg_2$, and $j\in\Ig_3$. We begin with the
  first of them where $\J_j=\{m_j\}$ for some $1\leq m_j\leq N$. Set  then
\begin{equation}
\label{eq:83}
    u_j = \frac{1}{2}\chi(|\cdot-\eta_{m_j}|/\epsilon)\sum_{
    \begin{subarray}{c}\strut
      (k,m)\in[\J_j^c]^2 \\
      k\neq m
    \end{subarray}}
[v_2(\cdot,\eta_k,\eta_m)- C^j_{km}]\,,
\end{equation}
where $\chi$ is given by \eqref{eq:25} and
\begin{equation}
\label{eq:84}
  C_{km}^j =v_2(z,\eta_i,\eta_k) \text{ wherein } |v_2(z,\eta_i,\eta_k)|= \min_{x\in\Ug_j^2}  |v_2(x,\eta_i,\eta_k)|\,.
\end{equation}
Note that the above definition guarantees that $u_j\in H^1_0(\Omega)$, even
in cases where $ \Ug_j^2\cap\partial\Omega\neq\emptyset$. Furthermore, for $1\leq n\leq N$
\begin{equation}
\label{eq:85}
  (u-u_j)|_{B_n} =
  \begin{cases}
    C_{m_j} & n=m_j \\
    0 & \text{otherwise}
  \end{cases}\,.
\end{equation}

We now prove
\begin{lemma}
  There exists $C(\Omega,f,\lambda,\Lambda)>0$ such that
  \begin{equation}
 \label{eq:86}
 \Eb_f\Big(\sum_{j:\,m_j\in\Kg_1}\|\nabla u_j\|_2^2\Big) \leq C\bar{\beta}^3\,.
  \end{equation}
\end{lemma}
\begin{proof}
 Since $d(B(\eta_{m_j},\epsilon),B(\eta_m,2\epsilon)\cup B(\eta_k,2\epsilon))>\epsilon$, we may use
\eqref{eq:77} to obtain that
\begin{displaymath}
  \|v_2(\cdot,\eta_k,\eta_m)- C^j_{km}\|_{L^\infty(B(\eta_{m_j},2\epsilon)}\leq C\frac{\epsilon^7}{|\eta_k-\eta_m|^3} \Big[ \frac{1}{|\eta_{m_j}-\eta_k|^3}
    +  \frac{1}{|\eta_{m_j}-\eta_m|^3}\Big](1+d_{km})^{1/2}\,,
\end{displaymath}
where $d_{km}$ is given, for $k\neq m$, by
\begin{displaymath}
  d_{km} = (|\ln \delta_k|+|\ln
\delta_m|){\mathbf 1}_{B(\eta_k,4\epsilon)}(\eta_m)\,.
\end{displaymath}
By the above
and \eqref{eq:72} we then obtain that
\begin{subequations}
\label{eq:87}
  \begin{equation}
  \|\nabla u_j\|_2\leq  C\tilde{u}_1(m_j)\,,
\end{equation}
where
\begin{equation}
  \tilde{u}_1(m_j)\overset{def}{=}\sum_{
    \begin{subarray}{c}\strut
      (k,m)\in[\J_j^c]^2 \\
      k\neq m
    \end{subarray}} \frac{\epsilon^{15/2}}{|\eta_m-\eta_k|^3}\frac{(1+ d_{km})^{1/2}}{|\eta_{m_j}-\eta_k|^3}
    \end{equation}
\end{subequations}
We now write,  in view of \eqref{eq:87}
\begin{multline*}
 \Eb_f\Big(\sum_{j:\,m_j\in\Kg_1}\|\nabla u_j\|_2^2\Big) \leq
 C\Eb_f\Big(\sum_{j:\,m_j\in\Kg_1}|\tilde{u}_1(m_j)|^2\Big) \\
    \leq C\Eb_f\Big(\sum_{i=1}^N|\tilde{u}_1(i)|^2\Big)= CN\Eb_f\big(|\tilde{u}_1(1)|^2\big)\,.
 \end{multline*}
As (cf. \cite[Eq. 3.7-3.15]{al13})
  \begin{multline}
\label{eq:88}
 N\Eb_f\big(|\tilde{u}_1(1)|^2\big)\leq
 C\Big[\bar{\beta}^3\int_{\Omega^3}\frac{\epsilon^6(1+d_{23}) }{|\eta_2-\eta_3|^6|\eta_1-\eta_2|^6}f_3(\eta_1,\eta_2,\eta_3) \,d\eta_1d\eta_2d\eta_3\\   
+\bar{\beta}^4\int_{\Omega^4}\Big[ \frac{1}{|\eta_2-\eta_1|^6}
    +  \frac{1}{|\eta_3-\eta_1|^6}\Big]\frac{\epsilon^3(1+d_{23})^{1/2}(1+d_{24})^{1/2}}{|\eta_2-\eta_4|^3|\eta_2-\eta_3|^3}
  f_4(\eta_1,\ldots,\eta_4) \,d\eta_1\cdots d\eta_4\, +\\ 
\bar{\beta}^5\int_{\Omega^5}\frac{(1+d_{23})^{1/2}(1+d_{45})^{1/2}}{|\eta_2-\eta_3|^3|\eta_4-\eta_5|^3|\eta_1-\eta_2|^3|\eta_1-\eta_4|^3}
    f_5(\eta_1,\ldots,\eta_5)\,d\eta_1\cdots d\eta_5\Big] \,,
\end{multline}
we obtain \eqref{eq:86} by \eqref{eq:6} and \eqref{eq:7}. 
 \end{proof}

Next consider the case where
\begin{displaymath}
  |\J_j|=2 \,,
\end{displaymath}
where we set $\J_j=\{m_{j1},m_{j2}\}$. Let further $\zeta_j\in
C^1(\Omega,[0,1])$ denote the cutoff function satisfying
\begin{equation}
  \label{eq:89}
\zeta_j(x) =
\begin{cases}
  1 & x\in\Ug_j^1 \\
  0 & x\in\Omega\setminus\Ug_j^2 \,,
\end{cases}
\quad |\nabla\zeta_j|\leq \frac{C}{\epsilon} \,.
\end{equation}
Then, set
\begin{equation}
  \label{eq:90}
u_j = -\zeta_j \Big[\sum_{i\in\J_j^c} [\phi_1(\cdot,\eta_i)-C^j_i]+\frac{1}{2}\sum_{
\begin{subarray}{c}\strut
  (k,m)\in[\J_j^c]^2 \\
  k\neq m
\end{subarray}}
    [v_2(\cdot,\eta_k,\eta_m)-
    C_{km}^j] \Big]
\end{equation}
where $C_i^j$ is given by
\begin{displaymath}
  C_i^j=\phi_1(z,\eta_i) \text{ wherein } |\phi_1(z,\eta_i)| = \min_{x\in \Ug_j^2}|\phi_1(x,\eta_i)|\,,
\end{displaymath}
and $C_{km}^j$ by \eqref{eq:84}. Note that by the above definition
$u_j\in H^1_0(\Omega)$, and furthermore for $1\leq n\leq N$, 
\begin{equation}
\label{eq:91}
  (u_j-u)|_{B_n} =
  \begin{cases}
  -\big(\phi-\psi_2(\cdot,\eta_{m_{j1}},\eta_{m_{j2}})\big)\big|_{B_n}+\tilde{C}_n =C_n &
  n\in\J_j  \\
0 & \text{otherwise}
  \end{cases} \,.
\end{equation}

We now prove
\begin{lemma}
  There exists $C(\Omega,f)>0$ such that
  \begin{equation}
\label{eq:92}
 \Eb_f\Big(\sum_{j:|J_j|=2}\|\nabla u_j\|_2^2\Big) \leq C\bar{\beta}^3\,.
  \end{equation}
\end{lemma}
\begin{proof}
  We use \eqref{eq:56} to obtain that
\begin{displaymath}
  \|\phi_1(\cdot,\eta_i)-C^j_i\|_{L^\infty(\Ug_{j}^2)}\leq C\epsilon^4 \Big[\frac{1}{|\eta_{m_{j1}}-\eta_i|^3}
    +  \frac{1}{|\eta_{m_{j2}}-\eta_i|^3}\Big]  \quad
    i\in\J_j^c \,.
\end{displaymath}
Hence, using the above and \eqref{eq:56} once again,
\begin{multline}
\label{eq:93}
   \|\nabla(\zeta_j[\phi_1(\cdot,\eta_i)-C^j_i])\|_{L^2(\Ug_{j}^2)} \leq \\C\epsilon^{9/2} \Big[\frac{1}{|\eta_{m_{j1}}-\eta_i|^3}
    +  \frac{1}{|\eta_{m_{j2}}-\eta_i|^3}\Big]  \quad
    i\in\J_j^c \,.
\end{multline}
Furthermore, by \eqref{eq:70},  we have that 
\begin{displaymath}
  \|v_2(\cdot,\eta_{m_{ji}},\eta_k)- C_{km}^j\|_{L^\infty(\Ug_{j}^2)}
  \leq C\frac{\epsilon^4}{|\eta_k-\eta_{m_{ji}}|^3}(1+d_{km})^{1/2}\quad i=1,2 \; k\in\J_j^c \,,
\end{displaymath} 
and hence, by the above and \eqref{eq:72}
\begin{multline}
\label{eq:94}
  \|\nabla(\zeta_j[v_2(\cdot,\eta_{m_{ji}},\eta_k)- C_{km_{ji}}^j])\|_{L^2(\Ug_{j}^2)}
  \leq \frac{C\epsilon^{9/2}}{|\eta_{m_{ji}}-\eta_k|^3}(1+d_{km})^{1/2} \\ \quad i=1,2\,, \;
  k\in\J_j^c \,.
\end{multline}
Finally, by \eqref{eq:77} we have for all
$(m,k)\in[\{1,\ldots,N\}\setminus\{m_{j1},m_{j2}\}]^2$ 
\begin{displaymath}
    \|\nabla(\zeta_j[v_2(\cdot,\eta_m,\eta_k)- C_{km}^j])\|_{L^2(\Ug_{j}^2)}
    \leq
    \frac{C\epsilon^{15/2}}{|\eta_m-\eta_k|^3}\Big[\frac{1}{d(\eta_m,\Ug_j^0)^3}+ 
    \frac{1}{d(\eta_k,\Ug_j^0)^3} \Big](1+d_{km})^{1/2}\,.
\end{displaymath}
Then, by the above, \eqref{eq:94},\eqref{eq:93}, \eqref{eq:90}
and \eqref{eq:87} we obtain that for some positive $C(\Omega,f,\Lambda,\lambda)$ 
\begin{subequations}
  \label{eq:95}
   \begin{equation}
     \|\nabla u_j\|_2\leq C\tilde{u}_2(\J_j)
  \end{equation}
in which
\begin{equation}
\tilde{u}_2(\J_j)=\tilde{u}_1(m_{j1}) + \tilde{u}_1(m_{j1})+\epsilon^{9/2}\sum_{
      i\in\J_j^c}
  \frac{1}{d(\eta_i,\Ug_j^0)^3}\,,
\end{equation}
\end{subequations}
wherein $\tilde{u}_1(n)$ is given by (\ref{eq:87}b).

We now write
\begin{multline*}
  \Eb_f\Big(\sum_{j:|\J_j|=2}\|\nabla u_j\|_2^2\Big) \leq C\Eb_f\Big(\sum_{m_{j1},m_{j2}\in\Kg_2}|\tilde{u}_2(m_{j1},m_{j2})|^2\Big)
\\\leq CN(N-1)\Eb_f\Big(|\tilde{u}_2(1,2)|^2\,\Big|\,|\eta_2-\eta_1|\leq4\epsilon \Big)\,.
\end{multline*}
Consequently, with the aid of  \eqref{eq:88} and \eqref{eq:86} we obtain
\begin{multline*}
  \Eb_f\Big(\sum_{j:|\J_j|=2}\|\nabla u_j\|_2^2\Big)
  \leq 
  C\bar{\beta}^3\int_{\Omega\times B(\eta_1,4\epsilon)\times\Omega}\frac{1}{|\eta_3-\eta_1|^6}
 f_3(\eta_1,\eta_2,\eta_3)\,d\eta_1d\eta_2d\eta_3 + \\ 
C\bar{\beta}^3N\int_{\Omega\times B(\eta_1,4\epsilon)\times\Omega^2}\frac{1}{|\eta_3-\eta_1|^3}\frac{1}{|\eta_4-\eta_1|^3}
  f_4(\eta_1,\eta_2,\eta_3)\,d\eta_1\cdots d\eta_4 + C\bar{\beta}^3\,,
\end{multline*}
which easily yields \eqref{eq:92} in view of \eqref{eq:6} and \eqref{eq:7} .
\end{proof}

Finally, we consider the case $|J_j|=K(j)\geq3$, and let
\begin{displaymath}
  \J_j=\{m_{j1},\ldots,m_{jK}\}\,.
\end{displaymath}
In this case we set
\begin{equation}
  \label{eq:96}
u_j=-\zeta_j\kappa_j\Big\{\bar{\phi}-(\bar{\phi})_{\Ug_j^2} + \sum_{i=1}^N \Big[
\phi_1(\cdot,\eta_i) -C_i^j+
\frac{1}{2}\sum_{
  \begin{subarray}{c}\strut
    k=1 \\
    k\neq i
  \end{subarray}}^N v_2(\cdot,\eta_i,\eta_k) -C_{ik}^j\Big]\Big\}\,,
\end{equation}
where $\kappa_j$ is the minimizer of \eqref{eq:20}, i.e.,
\begin{displaymath}
   \|\nabla\kappa_j\|_2^2 = C_K(\eta_{m_{j1}},\ldots,\eta_{m_{jK}}) \,.
\end{displaymath}
By the maximum principle $\|\kappa_j\|_\infty=1$. 
Note that by the definition of $\kappa_j$ it then follows that
$u_j\in H^1_0(\Omega)$. Furthermore, for all $1\leq n\leq N$,
\begin{equation}
\label{eq:97}
  (u_j-u)|_{\partial B_n} =
  \begin{cases}
  -\phi|_{\partial B_n}+\tilde{C}_n =C_n &
  n\in\J_j  \\
0 & \text{otherwise}
  \end{cases} \,.
\end{equation}

We now prove
\begin{lemma}
   There exists $C(\Omega,f,\lambda,\Lambda)>0$ such that
  \begin{equation}
\label{eq:98}
 \Eb_f\Big(\sum_{j:|J_j|\geq3}\|\nabla u_j\|_2^2\Big) \leq C\bar{\beta}^{5/2}\,.
  \end{equation}
\end{lemma}
\begin{proof}
{\em Step 1:} Estimate $\|\nabla u_j\|_2$.
\\

 By \eqref{eq:23} and the fact that $\|\nabla\bar{\phi}\|_\infty\leq C$ we have, 
\begin{displaymath}
  \|(\bar{\phi}-(\bar{\phi})_{\Ug_j^2})\nabla(\kappa_j\zeta_j)\|_2\leq
  Cl_K\epsilon^{3/2} \Big[\sum_{k=1}^K (1+|\ln \delta_{m_{jk}}|)\Big]^{1/2} \,,
\end{displaymath}
where
\begin{equation}
\label{eq:99}
  l_K=\min (K,1/\epsilon)\,.
\end{equation}
Hence, 
\begin{equation}
\label{eq:100}
\big\|\nabla\big(\kappa_j\zeta_j(\bar{\phi}-(\bar{\phi})_{\Ug_j^2})\big)\big\|_2\leq
Cl_K\epsilon^{3/2} \Big[\sum_{k=1}^K (1+|\ln \delta_{m_{jk}}|)\Big]^{1/2} \,. 
\end{equation}
We note that, had we managed to eliminate $l_K$ from \eqref{eq:100} and
the estimates below, we could have replace \eqref{eq:7} by the much
weaker assumptions on the marginal probability densities made in
\cite{al13} (that the first five marginal probability densities are
bounded). 

We now turn to estimate the $H^1$ norm of the first sum on the
right-hand-side of \eqref{eq:96}. By \eqref{eq:48}, for all $i\in\J_j$
we have,
\begin{displaymath}
  \|[\phi_1(\cdot,\eta_i)-(\phi_1(\cdot,\eta_i))_{\Ug_j^2}]\nabla(\zeta_j\kappa_j)\|_2 \leq   C\epsilon^{3/2} \Big[\sum_{k=1}^K (1+|\ln \delta_{m_{jk}}|)\Big]^{1/2} \,.
\end{displaymath}
Using the above and \eqref{eq:50} then yields  for all $i\in\J_j$
\begin{equation}
\label{eq:101}
  \|\nabla(\phi_1(\cdot,\eta_i)\zeta_j\kappa_j)\|_2 \leq   C\epsilon^{3/2} \Big[\sum_{k=1}^K (1+|\ln \delta_{m_{jk}}|)\Big]^{1/2} \,.
\end{equation}
For $i\not\in\J_j$ we have by \eqref{eq:56} that
\begin{displaymath}
  \|\phi_1(\cdot,\eta_i)-(\phi_1(\cdot,\eta_i))_{\Ug_j^2}\|_{L^\infty(\Ug_j^2)}\leq  C\frac{l_K\epsilon^4}{d(\eta_i,\Ug_j^2)^3}
\end{displaymath}
and that
\begin{displaymath}
   \|\nabla\phi_1(\cdot,\eta_i)\|_{L^2(\Ug_j^2)}\leq C\frac{K^{1/2}\epsilon^{9/2}}{d(\eta_i,\Ug_j^2)^3}
\end{displaymath}
Hence
\begin{displaymath}
  \|\nabla(\kappa_j\zeta_j[\phi_1(\cdot,\eta_i)-(\phi_1(\cdot,\eta_i))_{\Ug_j^2}])\|_2\leq
  C\frac{l_K\epsilon^{9/2}}{d(\eta_i,\Ug_j^2)^3}
  \Big[\sum_{k=1}^K (1+|\ln \delta_{m_{jk}}|)\Big]^{1/2} \,,
\end{displaymath}
which combined with \eqref{eq:101} then yields for all $1\leq i\leq N$
\begin{equation}
\label{eq:102}
  \|\nabla(\kappa_j\zeta_j[\phi_1(\cdot,\eta_i)-(\phi_1(\cdot,\eta_i))_{\Ug_j^2}])\|_2\leq
  C\frac{l_k\epsilon^{9/2}}{[d(\eta_i,\Ug_j^2)+\epsilon]^3}\Big[\sum_{k=1}^K (1+|\ln 
  \delta_{m_{jk}}|)\Big]^{1/2} \,.
\end{equation}

We now turn to the estimate of the second sum on the right-hand-side
of \eqref{eq:96} we have, when $(i,k)\in\J_j\times\J_j$, by \eqref{eq:70}
and \eqref{eq:59},
\begin{multline}
\label{eq:103}
   \|\nabla(\kappa_j\zeta_j[v_2(\cdot,\eta_i,\eta_k)
   -(v_2(\cdot,\eta_i,\eta_k))_{\Ug_j^2}])\|_2 \leq\\
   C\frac{\epsilon^{9/2}}{|\eta_i-\eta_k|^3}[1+d_{ik}]^{1/2}
   \Big[\sum_{n=1}^K (1+|\ln \delta_{m_{jn}}|)\Big]^{1/2} \,. 
\end{multline}
When either
$i\not\in\J_j$ or $k\not\in\J_j$ (or both) we have, by \eqref{eq:70} and
\eqref{eq:72} (or \eqref{eq:77}), that
\begin{multline}
\label{eq:104}
  \|\nabla(\kappa_j\zeta_j[v_2(\cdot,\eta_i,\eta_k)
  -(v_2(\cdot,\eta_i,\eta_k))_{\Ug_j^2}])\|_2\leq
  C\frac{\epsilon^{15/2}l_K}{|\eta_i-\eta_k|^3}[1+ d_{ik}]^{1/2} \\\times 
\Big[\frac{1}{[d(\eta_i,\Ug_j^2)+\epsilon]^3}+ \frac{1}{[d(\eta_k,\Ug_j^2)+\epsilon]^3}\Big]
  \Big[\sum_{n=1}^K (1+|\ln \delta_{m_{jn}}|)\Big]^{1/2} \,.
\end{multline}
Combining the above with \eqref{eq:103} reveals that \eqref{eq:104} is
valid for all $1\leq i,k\leq N$ such that $i\neq k$. The estimate of
$\|\nabla u_j\|_2$ is then derived from \eqref{eq:100}, \eqref{eq:102},
\eqref{eq:104}, and the fact that by \eqref{eq:96} we have
\begin{multline}
  \label{eq:105}
  \|\nabla u_j\|_2^2\leq3
  \bigg(\big\|\nabla\big(\kappa_j\zeta_j(\bar{\phi}-(\bar{\phi})_{\Ug_j^2})\big)\big\|_2^2 +
  \Big\|\sum_{i=1}^N\nabla(\kappa_j\zeta_j[\phi_1(\cdot,\eta_i)-(\phi_1(\cdot,\eta_i))_{\Ug_j^2}])\Big\|_2^2 + \\
  \Big\|\sum_{
    \begin{subarray}{c}
      i,k=1 \\
      i\neq N
    \end{subarray}}^N\nabla(\kappa_j\zeta_j[v_2(\cdot,\eta_i,\eta_k) -(v_2(\cdot,\eta_i,\eta_k))_{\Ug_j^2}]) \Big\|_2^2\bigg)
\end{multline}
\\

{\em Step 2:} Prove that
\begin{equation}
\label{eq:106}
  \Eb_f\Big(\sum_{j:\,|J_j|\geq3}\Big\|\nabla\Big(\zeta_j\kappa_j[\bar{\phi}-(\bar{\phi})_{\Ug_j^2}]\Big)\Big\|_2^2
\Big) \leq C\bar{\beta}^{5/2}\,.
\end{equation}
\\

Let 
 \begin{displaymath}
   S_n=\Prod_{k=1}^n\Union_{j=1}^kB(\eta_j,4\epsilon)\cap\Omega \,.
 \end{displaymath}
To prove \eqref{eq:106} we need an estimate for the expectation of the
right-hand-side of \eqref{eq:100}. We thus write
\begin{multline*}
  \Eb_f\Big(\sum_{j:|\J_j|\geq3}l_K^2\sum_{k=1}^{K(j)} (1+|\ln \delta_{m_{jk}}|\Big)=\\\sum_{k=1}^N\Eb_f\big(l_{K(k)}^2[1+|\ln
  \delta_k|]{\mathbf 1}_{k\in\Ig_3}\big)=N\Eb_f\big(l_{K(1)}^2[1+|\ln
  \delta_1|]{\mathbf 1}_{1\in\Ig_3}\big)= \\N\Eb_f\big(l_{K(1)}^2[1+|\ln
  \delta_1|]{\mathbf 1}_{1\in\Ig_3\setminus\Ig_{M_\beta}}\big)+N\Eb_f\big(l_{K(1)}^2[1+|\ln
  \delta_1|]{\mathbf 1}_{1\in\Ig_{M_\beta}}\big)\,,
\end{multline*}
where $M_\beta=[\beta^{-1/4}]$, i.e., the integer part of $\beta^{-1/4}$.  We
now use the definition of $l_K$ in
\eqref{eq:99} to obtain, with the aid of \eqref{eq:7}
\begin{multline}
\label{eq:107}
  N\Eb_f\big(l_{K(1)}^2[1+|\ln
  \delta_1|]{\mathbf 1}_{1\in\Ig_3\setminus\Ig_{M_\beta}}\big)\leq\\ N^3\beta^{-1/2}\int_{\Omega\times S_2}  (1+|\ln
  \delta_1|)f_3(\eta_1,\eta_2,\eta_3)\,d\eta_1d\eta_2d\eta_3 \leq CN\bar{\beta}^{3/2}\,.
\end{multline}
Furthermore, as
\begin{multline*}
  N\Eb_f\big(l_{K(1)}^2[1+|\ln
  \delta_1|]{\mathbf 1}_{1\in\Ig_{M_\beta}}\big) \leq\\
  \frac{C \,N!}{(N-M_\beta-1)!\epsilon^2}\int_{\Omega\times S_{M_\beta}}  (1+|\ln
  \delta_1|)f_{M_\beta+1}(\eta_1,\ldots,\eta_{M_\beta+1})\,d\eta_1\cdots d\eta_{M_\beta+1}
\end{multline*}
we obtain from \eqref{eq:7}  that
\begin{equation}
\label{eq:108}
  N\Eb_f\big(l_{K(1)}^2[1+|\ln
  \delta_1|]{\mathbf 1}_{1\in\Ig_{M_\beta}}\big) \leq
  C\frac{M_\beta!}{\epsilon^2}(C_1\beta)^{M_\beta}\leq \frac{C}{\epsilon^2}(C_1\beta^{3/4})^{M_\beta}\,,
\end{equation}
where $C_1<C_0/e$ and $C_0$ is the same as in \eqref{eq:7}.  We may
now conclude \eqref{eq:106} from the above, \eqref{eq:6}, and
\eqref{eq:107}. 
\\

{\em Step 3:} Prove that 
\begin{equation}
  \label{eq:109}
 \Eb_f\Big(\sum_{j:\,|J_j|\geq3}\Big\|\sum_{i=1}^N\nabla(\kappa_j\zeta_j[\phi_1(\cdot,\eta_i)-C_i^j])\Big\|_2^2\Big) \leq C\bar{\beta}^{11/4}\,.
\end{equation}
\\

We first observe that
\begin{multline}
\label{eq:110}
  \Eb_f\Big(\sum_{j:\,|J_j|\geq3}\Big\|\sum_{i=1}^N\nabla(\kappa_j\zeta_j[\phi_1(\cdot,\eta_i)-C_i^j])\Big\|_2^2\Big) \\ =   
2N\Eb_f\Big(\sum_{j:\,|J_j|\geq3}\Big\|\nabla(\kappa_j\zeta_j[\phi_1(\cdot,\eta_N)-C_N^j])\Big\|_2^2\Big)
+
\\2N(N-1)\Eb_f\Big(\sum_{j:\,|J_j|\geq3}\Big\langle\nabla(\kappa_j\zeta_j[\phi_1(\cdot,\eta_N)-C_N^j]),\nabla(\kappa_j\zeta_j[\phi_1(\cdot,\eta_{N-1})-C_{N-1}^j])\Big\rangle
\Big)
\end{multline}
By \eqref{eq:102} we then have
\begin{multline}
\label{eq:111}
  N\Eb_f\Big(\sum_{j:\,|J_j|\geq3}\Big\|\nabla(\kappa_j\zeta_j[\phi_1(\cdot,\eta_N)-C_N^j])\Big\|_2^2\Big)
  \\ \leq C\bar{\beta}^2\epsilon^3\Eb_f\Big([1+|\ln
  \delta_1|]{\mathbf 1}_{1\in\Ig_3}\frac{l_{K(1)}^2}{[d(\eta_N,\Ug_1^2)+\epsilon]^6}\Big)
  \,,
\end{multline}
where $l_K$ is given by \eqref{eq:99}  and $\Ug_1^2$ is given by
\eqref{eq:80} (with $j=1$). We now
write
\begin{multline*}
 \Eb_f\Big([1+|\ln \delta_1|]{\mathbf 1}_{1\in\Ig_3}\frac{l_K^2}{[d(\eta_N,\Ug_1^2)+\epsilon]^6}\Big) \\ =
   \Eb_f\Big([1+|\ln
   \delta_1|]{\mathbf 1}_{1\in\Ig_3\setminus\Ig_{M_\beta}}\frac{l_K^2}{[d(\eta_N,\Ug_1^2)+\epsilon]^6}\Big) \\+ 
  \Eb_f\Big([1+|\ln
  \delta_1|]{\mathbf 1}_{1\in\Ig_{M_\beta}}\frac{l_K^2}{[d(\eta_N,\Ug_1^2)+\epsilon]^6}\Big) 
\end{multline*}
For the first term on the right-hand-side we have  
\begin{multline}
\label{eq:112}
   \Eb_f\Big([1+|\ln
   \delta_1|]{\mathbf 1}_{1\in\Ig_3\setminus\Ig_{M_\beta}}\frac{l_K^2}{[d(\eta_N,\Ug_1^2)+\epsilon]^6}\Big) \leq\\
   \frac{C}{\beta^{1/2}}N^2\Big[ \int_{\Omega\times S_2\times B(\eta_1,5\beta^{-1/4}\epsilon)}\frac{[1+|\ln
   \delta_1|]}{\epsilon^6}f_4(\eta_1,\eta_2,\eta_3,\eta_4) \, d\eta_1\cdots d\eta_4 + \\ \int_{\Omega\times S_2\times \Omega\setminus B(\eta_1,5\beta^{-1/4}\epsilon)}\frac{[1+|\ln
   \delta_1|]}{[|\eta_4-\eta_1|-\beta^{-1/4}\epsilon]^6}f_4(\eta_1,\eta_2,\eta_3,\eta_4) \, d\eta_1\cdots d\eta_4
 \Big] \leq C\frac{\bar{\beta}^{3/4}}{\epsilon^3} \,. 
\end{multline}
Note that ${\rm Diam}(\Ug_1^2)\leq4\beta^{-1/4}\epsilon$ whenever
$1\in\Ig_3\setminus\Ig_{M_\beta}$.  Furthermore,
\begin{displaymath}
    \Eb_f\Big([1+|\ln \delta_1|]{\mathbf
      1}_{1\in\Ig_{M_\beta}}\frac{l_K^2}{[d(\eta_N,\Ug_1^2)+\epsilon]^6}\Big) \leq
\frac{C}{\epsilon^6} \Eb_f\Big([1+|\ln \delta_1|]{\mathbf
      1}_{1\in\Ig_{M_\beta}}l_K^2\Big)\,,
  \end{displaymath}
and hence by \eqref{eq:108} we obtain that
\begin{equation}
  \label{eq:113}
\Eb_f\Big([1+|\ln \delta_1|]{\mathbf
      1}_{1\in\Ig_{M_\beta}}\frac{l_K^2}{[d(\eta_N,\Ug_1^2)+\epsilon]^6}\Big) \leq \frac{C}{\epsilon^8}(C_1\beta^{3/4})^{M_\beta}\,.
\end{equation}
Combining the above with \eqref{eq:112} and \eqref{eq:111} yields, in
view of \eqref{eq:6}, for sufficiently small $\bar{\beta}$
\begin{equation}
  \label{eq:114}
N\Eb_f\Big(\sum_{j:\,|J_j|\geq3}\Big\|\nabla(\kappa_j\zeta_j[\phi_1(\cdot,\eta_N)-C_N^j])\Big\|_2^2\Big)
\leq  C\bar{\beta}^{11/4}\,.
\end{equation}

To estimate the second term on the right-hand-side of \eqref{eq:110}
we first note that by \eqref{eq:102}
\begin{multline*}
  \Big\langle\nabla(\kappa_j\zeta_j[\phi_1(\cdot,\eta_N)-C_N^j]),\nabla(\kappa_j\zeta_j[\phi_1(\cdot,\eta_{N-1})-C_{N-1}^j])\Big\rangle\leq  \\ 
\frac{Cl_K^2\epsilon^9}{[d(\eta_N,\Ug_j^2)+\epsilon]^3[d(\eta_{N-1},\Ug_j^2)+\epsilon]^3}\Big[\sum_{m=1}^K (1+|\ln 
  \delta_{m_{jk}}|)\Big] \,.
\end{multline*}
Hence,
\begin{multline*}
  N^2\Eb_f\Big(\Big\langle\nabla(\kappa_j\zeta_j[\phi_1(\cdot,\eta_N)-C_N^j]),\nabla(\kappa_j\zeta_j[\phi_1(\cdot,\eta_{N-1})-C_{N-1}^j])\Big\rangle\Big) \leq \\
C\bar{\beta}^3\Eb_f\Big([1+|\ln
  \delta_1|]{\mathbf 1}_{1\in\Ig_3}\frac{l_K}{[d(\eta_N,\Ug_1^2)+\epsilon]^3}\frac{l_K}{[d(\eta_{N-1},\Ug_1^2)+\epsilon]^3}\Big)
\end{multline*}
Then we write, as before,
\begin{multline}
\label{eq:115}
  \Eb_f\Big([1+|\ln
  \delta_1|]{\mathbf
    1}_{1\in\Ig_3}\frac{l_K}{[d(\eta_N,\Ug_1^2)+\epsilon]^3}\frac{l_K}{[d(\eta_{N-1},\Ug_1^2)+\epsilon]^3}\Big) = \\ \Eb_f\Big([1+|\ln
  \delta_1|]{\mathbf
    1}_{1\in\Ig_3\setminus\Ig_{M_\beta}}\frac{l_K}{[d(\eta_N,\Ug_1^2)+\epsilon]^3}\frac{l_K}{[d(\eta_{N-1},\Ug_1^2)+\epsilon]^3}\Big) + \\ 
\Eb_f\Big([1+|\ln
  \delta_1|]{\mathbf
    1}_{1\in\Ig_{M_\beta}} \frac{l_K}{[d(\eta_N,\Ug_1^2)+\epsilon]^3}\frac{l_K}{[d(\eta_{N-1},\Ug_1^2)+\epsilon]^3}\Big)\,.
\end{multline}
For the first term on the right-hand-side we have, by \eqref{eq:6} and
\eqref{eq:7},
\begin{multline}
\label{eq:116}
  \Eb_f\Big([1+|\ln
  \delta_1|]{\mathbf
    1}_{1\in\Ig_3\setminus\Ig_{M_\beta}}\frac{l_K}{[d(\eta_N,\Ug_1^2)+\epsilon]^3}\frac{l_K}{[d(\eta_{N-1},\Ug_1^2)+\epsilon]^3}\Big) \leq \\
\frac{C}{\beta^{1/2}}N^2\Big[ \int_{\Omega\times S_2\times [B(\eta_1,5\beta^{-1/4}\epsilon)]^2}\frac{[1+|\ln
   \delta_1|]}{\epsilon^6}f_5(\eta_1,\ldots,\eta_5) \, d\eta_1\cdots d\eta_5 + \\ \int_{\Omega\times S_2\times [\Omega\setminus B(\eta_1,5\beta^{-1/4}\epsilon)]^2}\frac{[1+|\ln
   \delta_1|]}{[|\eta_4-\eta_1|-\beta^{-1/4}\epsilon]^3[|\eta_5-\eta_1|-\beta^{-1/4}\epsilon]^3}f_5(\eta_1,\ldots,\eta_5) \, d\eta_1\cdots
 d\eta_5 + \\ 
2 \int_{\Omega\times S_2\times B(\eta_1,5\beta^{-1/4}\epsilon)\times\Omega\setminus B(\eta_1,5\beta^{-1/4}\epsilon)}\frac{[1+|\ln
   \delta_1|]}{[|\eta_4-\eta_1|-\beta^{-1/4}\epsilon]^3\epsilon^3}f_5(\eta_1,\ldots,\eta_5) \, d\eta_1\cdots
 d\eta_5 
 \Big] \\\leq C\,. 
\end{multline}
The second term on the right-hand-side of \eqref{eq:115} can be
bounded as in \eqref{eq:113}
\begin{displaymath}
  \Eb_f\Big([1+|\ln
  \delta_1|]{\mathbf
    1}_{1\in\Ig_{M_\beta}}
  \frac{l_K}{[d(\eta_N,\Ug_1^2)+\epsilon]^3}\frac{l_K}{[d(\eta_{N-1},\Ug_1^2)+\epsilon]^3}\Big)\leq \frac{C}{\epsilon^8}(C_1\beta^{3/4})^{M_\beta}\,.
\end{displaymath}
Combining the above with \eqref{eq:116}, \eqref{eq:115},
\eqref{eq:114}, and \eqref{eq:110} yields \eqref{eq:109}.
\\

{\em Step 4:} Prove that
\begin{equation}
\label{eq:117}
   \Eb_f\Big(\sum_{j:\,|J_j|\geq3}\Big\|\sum_{
    \begin{subarray}{c}
      i,k=1 \\
      i\neq k
    \end{subarray}}^N\nabla(\kappa_j\zeta_j[v_2(\cdot,\eta_i,\eta_k)
   -(v_2(\cdot,\eta_i,\eta_k))_{\Ug_j^2}])\Big\|_2^2\Big)\leq C\bar{\beta}^{15/4}\,.
\end{equation}
\\

 We begin by writing (cf. \cite[Eq. (3.6)-(3.9)]{al13})
\begin{multline}
\label{eq:118}
  \Eb_f\Big(\sum_{j:\,|J_j|\geq3}\Big\|\sum_{
    \begin{subarray}{c}
      i,k=1 \\
      i\neq k
    \end{subarray}}^N\nabla(\kappa_j\zeta_j[v_2(\cdot,\eta_i,\eta_k)
   -C_{ik}^j])\Big\|_2^2\Big)\leq\\
   2N^2\Eb_f\Big(\sum_{j:\,|J_j|\geq3}\| \nabla(\kappa_j\zeta_j[v_2(\cdot,\eta_N,\eta_{N-1})
   -C_{N,N-1}^j)\|_2^2\Big)+ \\4N^3\Big|\Eb_f\Big(\sum_{j:\,|J_j|\geq3}\langle\nabla(\kappa_j\zeta_j[v_2(\cdot,\eta_N,\eta_{N-1})
   -C_{N,N-1}^j]),\nabla(\kappa_j\zeta_j[v_2(\cdot,\eta_N,\eta_{N-2})
   -C_{N,N-2}^j]\rangle\Big)\Big|+\\N^4\Big|\Eb_f\Big(\sum_{j:\,|J_j|\geq3}\langle\nabla(\kappa_j\zeta_j[v_2(\cdot,\eta_N,\eta_{N-1})
   -C_{N,N-1}^j]),\nabla(\kappa_j\zeta_j[v_2(\cdot,\eta_N,\eta_{N-2})
   -C_{N-2,N-3}^j]\rangle\Big)\Big|
\end{multline}
By \eqref{eq:104} we then obtain that
\begin{multline}
\label{eq:119}
  2N^2\Eb_f\Big(\sum_{j:\,|J_j|\geq3}\|\nabla(\kappa_j\zeta_j[v_2(\cdot,\eta_N,\eta_{N-1})
   -C_{N,N-1}^j])\|_2^2\Big) \leq
   \\ C\bar{\beta}^3\epsilon^6\Eb_f\bigg([1+|\ln \delta_1|]{\mathbf
    1}_{1\in\Ig_3}\Big[ \frac{1}{[d(\eta_4,\Ug_1^2)+\epsilon]^6}+ \frac{1}{[d(\eta_5,\Ug_1^2)+\epsilon]^6}\Big]\frac{1+d_{45}}{|\eta_4-\eta_5|^6}\bigg) \,.
\end{multline}
For convenience we have replaced in the above the indices $N$ and
$N-1$ by $4$ and $5$ (the statement remains accurate as all inclusions
are identical). We shall apply a similar change of indices in the sequel
without referring to that explicitly. We now write as above, using the
symmetry of $f_N$,
\begin{multline*}
  \Eb_f\bigg([1+|\ln \delta_1|]{\mathbf
    1}_{1\in\Ig_3}\Big[ \frac{1}{[d(\eta_4,\Ug_1^2)+\epsilon]^6}+ \frac{1}{[d(\eta_5,\Ug_1^2)+\epsilon]^6}\Big]\frac{1+d_{45}}{|\eta_4-\eta_5|^6}\bigg) \leq \\
 2\Eb_f\bigg([1+|\ln \delta_1|]{\mathbf
    1}_{1\in\Ig_3\setminus\Ig_{M_\beta}}\frac{1}{[d(\eta_4,\Ug_1^2)+\epsilon]^6}\frac{1+d_{45}}{|\eta_4-\eta_5|^6}\bigg) + \\
 \Eb_f\bigg([1+|\ln \delta_1|]{\mathbf
    1}_{1\in\Ig_{M_\beta}}\frac{1}{[d(\eta_4,\Ug_1^2)+\epsilon]^6}\frac{1+d_{45}}{|\eta_4-\eta_5|^6}\bigg) \,.
\end{multline*}
For the first term we have
\begin{multline}
\label{eq:120}
\Eb_f\bigg([1+|\ln \delta_1|]{\mathbf
  1}_{1\in\Ig_3\setminus\Ig_{M_\beta}}l_K^2\frac{1}{[d(\eta_4,\Ug_1^2)+\epsilon]^6}\frac{1+d_{45}}{|\eta_4-\eta_5|^6}\bigg)
\leq \\\frac{C}{\bar{\beta}^{1/2}}N^2\Big[ \int_{\Omega\times S_2\times
  B(\eta_1,5\beta^{-1/4}\epsilon)\cap\Omega\times\Omega}\frac{[1+|\ln
  \delta_1|]}{\epsilon^6}\frac{1+d_{45}}{|\eta_4-\eta_5|^6}f_5(\eta_1,\ldots,\eta_5) \,
d\eta_1\cdots d\eta_5 + \\ \int_{\Omega\times S_2\times \Omega\setminus B(\eta_1,5\beta^{-1/4}\epsilon)\times\Omega}[1+|\ln
\delta_1|]\frac{1}{[|\eta_4-\eta_1|-\beta^{-1/4}\epsilon]^6}\frac{1+d_{45}}{|\eta_4-\eta_5|^6}
f_5 \, d\eta_1\cdots d\eta_5 \\\leq C\frac{\bar{\beta}^{3/4}|\ln \epsilon|}{\epsilon^6}\,.
\end{multline}
For the second term we have, as in \eqref{eq:108}, 
\begin{multline*}
   \Eb_f\bigg([1+|\ln \delta_1|]{\mathbf
    1}_{1\in\Ig_{M_\beta}}\frac{1}{[d(\eta_4,\Ug_1^2)+\epsilon]^6}\frac{1+d_{45}}{|\eta_4-\eta_5|^6}\bigg) \leq \\
\frac{C}{\epsilon^{12}}\Eb_f\bigg([1+|\ln \delta_1|]{\mathbf
    1}_{1\in\Ig_{M_\beta}}(1+d_{45})\bigg) \leq
\frac{C}{\epsilon^7|\ln\epsilon|}(C_1\beta^{3/4})^{M_\beta-1}\,.
\end{multline*}
Combining the above with \eqref{eq:120} yields, in view of
\eqref{eq:6}, for sufficiently small $\beta$,
\begin{displaymath}
  \Eb_f\bigg([1+|\ln \delta_1|]{\mathbf
    1}_{1\in\Ig_3}\frac{1}{[d(\eta_4,\Ug_1^2)+\epsilon]^6}\frac{1+d_{45}}{|\eta_4-\eta_5|^6}\bigg) \leq C\frac{\bar{\beta}^{3/4}}{\epsilon^6} 
\end{displaymath}
Hence, by \eqref{eq:119} we have
\begin{equation}
\label{eq:121}
  2N^2\Eb_f\Big(\sum_{j:\,|J_j|\geq3}\|\nabla(\kappa_j\zeta_j[v_2(\cdot,\eta_N,\eta_{N-1})
   -C_{N,N-1}^j]\|_2^2\Big) \leq C \bar{\beta}^{15/4}\,.
\end{equation}

For the second term on the right-hand-side of \eqref{eq:118} we have
\begin{multline}
  N^3\Big|\Eb_f\Big(\sum_{j:\,|J_j|\geq3}\langle\nabla(\kappa_j\zeta_j[v_2(\cdot,\eta_N,\eta_{N-1})
   -C_{N,N-1}^j]),\nabla(\kappa_j\zeta_j[v_2(\cdot,\eta_N,\eta_{N-2})
   -C_{N,N-2}^j]\rangle\Big)\Big| \leq \\ C\bar{\beta}^4\epsilon^3\Eb_f\bigg(l_K^2[1+|\ln \delta_1|]{\mathbf
    1}_{1\in\Ig_3}\frac{1}{[d(\eta_4,\Ug_1^2)+\epsilon]^6}\frac{1+d_{45} + d_{46}}{|\eta_4-\eta_5|^3|\eta_4-\eta_6|^3}\bigg) 
\end{multline}
Following precisely the same steps as in the derivation of
\eqref{eq:121} leads to
\begin{multline}
\label{eq:122}
  N^3\Big|\Eb_f\Big(\sum_{j:\,|J_j|\geq3}\langle\nabla(\kappa_j\zeta_j[v_2(\cdot,\eta_N,\eta_{N-1})
   -C_{N,N-1}^j]),\nabla(\kappa_j\zeta_j[v_2(\cdot,\eta_N,\eta_{N-2})
   -C_{N,N-2}^j]\rangle\Big)\Big| \\\leq C\bar{\beta}^{19/4}\ln^2\epsilon\leq C \bar{\beta}^{17/4}\,.
\end{multline}
Finally, in a similar manner, we obtain that
\begin{multline}
  \label{eq:123}
N^4\Big|\Eb_f\Big(\sum_{j:\,|J_j|\geq3}\langle\nabla(\kappa_j\zeta_j[v_2(\cdot,\eta_N,\eta_{N-1})
   -C_{N,N-1}^j]),\nabla(\kappa_j\zeta_j[v_2(\cdot,\eta_N,\eta_{N-2})
   -C_{N-2,N-3}^j]\rangle\Big)\Big|\\\leq C( \bar{\beta}^{13/2}\ln^4\epsilon +
   \bar{\beta}^5\ln^2\epsilon)\leq  C \bar{\beta}^{9/2}\,.
\end{multline}
Combining the above with \eqref{eq:122} and \eqref{eq:119} yields \eqref{eq:117}.

The lemma now follows from \eqref{eq:105}, \eqref{eq:117},
\eqref{eq:109}, and \eqref{eq:106}. 
 \end{proof}

We may now combine \eqref{eq:86}, \eqref{eq:92}, and \eqref{eq:98}
into the following statement
\begin{proposition}
   Under the assumptions of Theorem \ref{thm:1} there exists $C(\Omega)>0$ such that
     \begin{equation}
 \label{eq:124}
\Eb_f\Big( \|\nabla u\|_2^2\Big) \leq C\bar{\beta}^{5/2} \,.
  \end{equation}
\end{proposition}

We can now establish 
\begin{proposition}
  Let $\phi$ denote the (weak) solution of \eqref{eq:1}. Then, under the
  assumptions \eqref{eq:6} and \eqref{eq:9} we have that
  \begin{equation}
\label{eq:125}
    \| \Eb_f(\phi) - \bar{\phi}- N\Eb_f(\phi_1)\|_{1,2} \leq C\bar{\beta}^{5/2} \,.
  \end{equation}
\end{proposition}
\begin{proof}
  Let 
\begin{displaymath}
  V_2(x,\eta_1,\ldots,\eta_N)= \sum_{i=1}^N \sum_{
  \begin{subarray}{2}
    j=1 \\
    j\neq i
  \end{subarray}}^N v_2(x,\eta_i,\eta_j)\,.
\end{displaymath} 
Clearly,
\begin{displaymath}
  \|\nabla\Eb_f(V_2 )\|_2^2 =
\|N(N-1)\Eb_f\big(\nabla v_2\big)\|_2^2  \,.
\end{displaymath}
With the aid of \eqref{eq:72} we then obtain for any $z\in\Omega$
\begin{displaymath}
  \|\nabla\Eb_f(V_2 )\|_{L^2(B(z,\epsilon)\cap\Omega}^2\leq CN^4\Eb_f\big(\|\nabla v_2
  \|_{L^2(B(z,\epsilon)\cap\Omega}\big)^2\leq C\beta^4\epsilon^3\ln^4\epsilon\,.
\end{displaymath}
From which we readily obtain that
\begin{equation}
\label{eq:126}
  \|\nabla\Eb_f(V_2 )\|_2\leq C\beta^2\ln^2\epsilon\,.
\end{equation}

By \eqref{eq:124} we have
\begin{displaymath}
  \|\nabla\Eb_f(u)\|_2^2 \leq  \Eb_f(\|\nabla u\|_2^2)\leq C\bar{\beta}^{5/2}\,.
\end{displaymath}
The proposition now follows from the above, \eqref{eq:126}, 
\eqref{eq:78}, and Poincare inequality. 
\end{proof}

\section{Effective medium}
\label{sec:4}

To prove Theorem \ref{thm:1} we need to show that the estimate of $\Eb_f(\phi)$
provided by \eqref{eq:125} is a good approximation for the
solution of the steady-state heat equation in a continuous medium
whose conductivity is a function of both the  conductivity $a(x)$ 
and the volume fraction $\beta(x)$. Consider then the following problem
\begin{subequations}
 \label{eq:127}
 \begin{equation}
  \begin{cases}
   -\nabla\cdot(a_e(x)\nabla\phi_e) = 0 & \text{in } \Omega \\
   \phi=f & \text{on } \partial\Omega\,,
 \end{cases} 
\end{equation}
where 
\begin{equation}
  a_e(x)= a(x)[1 + \gamma(x)] \,,
\end{equation}
in which 
\begin{equation}
\|\gamma\|_\infty  \leq C\bar{\beta} \,,  
\end{equation} 
and
\begin{equation}
  a_e>\frac{1}{2} \,,
\end{equation} 
for all $x$ in $\Omega$.
\end{subequations}

For the solution of \eqref{eq:127} we prove  the following estimate
\begin{lemma}
  Let $\phi_e$ denote the unique solution of \eqref{eq:127}. Then
  \begin{equation}
    \label{eq:128}
\|\phi_e- \bar{\phi}+ \LL^{-1}\big(\nabla\cdot(a\gamma\nabla\bar{\phi})\big)\|_{1,2} \leq C
\bar{\beta}^2 \,.
  \end{equation}
In the above $\LL^{-1}$ denotes the inverse of $\LL$ in
$H^1_0(\Omega)$, i.e., for any $F\in H^{-1}(\Omega)$, $w=\LL^{-1}F$ is the unique
(weak) solution of
\begin{displaymath}
  \begin{cases}
    \LL w = F & \text{in } \Omega \\
    w=0 & \text{on } \partial\Omega 
  \end{cases}\,.
\end{displaymath}
\end{lemma}
\begin{proof}
 The proof is almost identical with the proof of \cite[Lemma 6.1] {al13}. Set
  \begin{displaymath}
    u_e = \phi_e- \bar{\phi}+\LL^{-1}\big(\nabla\cdot (a\gamma\nabla\bar{\phi})\big) \,.
  \end{displaymath}
Then,
\begin{displaymath}
  \begin{cases}
      -\nabla\cdot (a_e \nabla u_e) = -\nabla\cdot\big\{\gamma a\nabla\LL^{-1}\big(\nabla\cdot(\gamma a\nabla\bar{\phi})\big)\big\} & \text{in}
      \Omega \\
      u_e=0 & \text{on } \partial\Omega \,.
  \end{cases}
\end{displaymath}
Consequently, as $a_e>\lambda/2$ for sufficiently small $\bar{\beta}$ we have
that \cite{al13}
\begin{displaymath}
  \|\nabla u_e\|_2 \leq C
  \big\|\gamma a\nabla\LL^{-1}\big(\nabla\cdot(a\gamma\nabla\bar{\phi})\big)\big\|_2 \leq C\|\gamma\|_\infty^2\|\nabla\bar{\phi}\|_2 \,.
\end{displaymath}
From (\ref{eq:127}c) we then get  \eqref{eq:128}.
\end{proof}

We next show that $N\langle\phi_1\rangle$ can approximately be obtained by
applying $\LL^{-1}$ to $\nabla\cdot(a\gamma\nabla\bar{\phi})$ for an appropriate choice
of $\gamma$.
\begin{lemma}
Let $\phi_1$ be given by \eqref{eq:28}. Then, 
  \begin{equation}
    \label{eq:129}
  \Big\|N\Eb_f(\phi_1)+3\LL^{-1}\big(\nabla\cdot(a\beta\nabla\bar{\phi})\big) \Big\|_{1,2} \leq
  C\epsilon^{1/2} |\ln \epsilon|^{1/2}\bar{\beta}\,.
  \end{equation}
\end{lemma}
\begin{proof}
By \eqref{eq:19}, \eqref{eq:28}, and \eqref{eq:32} we have that
\begin{displaymath}
  \phi_1(x,\eta)= \int_{\partial B(\eta,\epsilon)} 
G(x,\xi)a(\xi)\Big[\frac{\partial\bar{\phi}}{\partial\nu}(\xi) + \frac{\partial\phi_0}{\partial\nu}(\xi,\eta)+\frac{\partial v_1}{\partial\nu}(\xi,\eta)\Big] \,ds_\xi \,.
\end{displaymath}
It can be easily verified from \eqref{eq:18}  that for every $x\in\Omega$,
\begin{displaymath}
  W_0^1(x)=\int_{\partial B(\eta,\epsilon)} 
G(x,\xi)a(\xi)\frac{\partial\bar{\phi}}{\partial\nu}(\xi) \,ds_\xi = \int_{B(\eta,\epsilon)} \nabla_\xi G(x,\xi)\cdot a(\xi)\nabla\bar{\phi}(\xi) \,d\xi \,.
\end{displaymath}
For the expectation we then obtain
\begin{displaymath}
  \Eb_f(W_0^1)= \int_{\Omega_\epsilon} \int_{B(\eta,\epsilon)} a(\xi)\nabla_\xi G(x,\xi)\cdot\nabla\bar{\phi}(\xi) \,d\xi\,
  f_1(\eta)\,d\eta\,.
\end{displaymath}
Interchanging  the order of integration then yields, by \eqref{eq:9},
\begin{equation}
\label{eq:130}
    N\Eb_f(W_0^1)=\int_\Omega\nabla_\xi G(x,\xi)\cdot a(\xi)\nabla\bar{\phi}(\xi)\int_{B(\xi,\epsilon)}Nf_1(\eta)\,d\eta  \,d\xi = -\LL^{-1}(\beta a\nabla\bar{\phi})\,,
\end{equation}
where use has been made of the fact that $f_1\equiv0$ in $\R^3\setminus\Omega_\epsilon$. 

Next we compute the expectation of
\begin{displaymath}
  W_0^2(x)=\int_{\partial B(\eta,\epsilon)} 
G(x,\xi)a(\xi)\frac{\partial\phi_0}{\partial\nu}(\xi) \,ds_\xi = \int_{\partial B(\eta,\epsilon)} 
G(x,\xi)a(\xi)[2\nabla\bar{\phi}(\eta)\cdot\nu-\epsilon C_a] \,ds_\xi \,,
\end{displaymath}
where we have used the definition of $\phi_0$ in
\eqref{eq:29}. Integration by parts then yields
\begin{subequations}
  \label{eq:143}
\begin{equation}
  W_0^2=\tilde{W}_1 + \tilde{W}_2 \,,
\end{equation}
where
\begin{equation}
  \tilde{W}_1(x) =\int_{B(\eta,\epsilon)}G(x,\xi)[2\nabla a(\xi)\cdot\nabla\bar{\phi}(\eta)-3a(\xi)C_a-C_a\nabla a(\xi)\cdot(\xi-\eta)]\,d\xi \,,
\end{equation}
and
\begin{equation}
  \tilde{W}_2(x)= \int_{B(\eta,\epsilon)}a(\xi)\nabla_\xi G(x,\xi)\cdot[2\nabla\bar{\phi}(\eta)-C_a(\xi-\eta)]\,d\xi \,.
\end{equation}
\end{subequations}
Let
\begin{displaymath}
  \tilde{g}=2\nabla a(\xi)\cdot\nabla\bar{\phi}(\eta)-3a(\xi)C_a-C_a\nabla a(\xi)\cdot(\xi-\eta) \,.
\end{displaymath}
Clearly,
\begin{displaymath}
  \Eb_f(  \tilde{W}_1) = \int_{\Omega_\epsilon} \int_{B(\eta,\epsilon)} G(x,\xi)\cdot\tilde{g}(\xi,\eta) \,d\xi\,
  f_1(\eta)\,d\eta\,.
\end{displaymath}
Interchanging  the order of integration then yields, as above,
\begin{equation}
\label{eq:131}
    N\Eb_f(\tilde{W}_1)= \LL^{-1}\Big(N\int_{B(\cdot,\epsilon)} \tilde{g}(\cdot,\eta)
    f_1(\eta)\,d\eta\Big)\,.
\end{equation}
By \eqref{eq:31} we have that
\begin{displaymath}
  \sup_{(\xi,\eta)\in\Omega\times B(\xi,\epsilon)}|\tilde{g}|\leq C\epsilon^\alpha \,.
\end{displaymath}
Consequently, we obtain that
\begin{equation}
  \label{eq:132}
\|N\Eb_f(\tilde{W}_1)\|_{1,2} \leq C\beta\epsilon^\alpha \,. 
\end{equation}

We now estimate 
\begin{displaymath}
   \Eb_f(  \tilde{W}_2) = \int_{\Omega_\epsilon} \int_{B(\eta,\epsilon)} a(\xi)\nabla_\xi G(x,\xi)\cdot[2\nabla\bar{\phi}(\eta)-C_a(\xi-\eta)] \,d\xi\,
  f_1(\eta)\,d\eta\,,
\end{displaymath}
from which we easily obtain that
\begin{displaymath}
    \Eb_f(\tilde{W}_2)= \int_{\Omega_\epsilon}  a(\xi)\nabla_\xi
    G(x,\xi)\cdot\int_{B(\xi,\epsilon)}[2\nabla\bar{\phi}(\eta)-C_a(\xi-\eta)]f_1(\eta) \,d\eta\,d\xi\,. 
\end{displaymath}
As
\begin{displaymath}
  \sup_{(\xi,\eta)\in\Omega\times
    B(\xi,\epsilon)}\big|2\big(\nabla\bar{\phi}(\eta)-\nabla\bar{\phi}(\xi)\big)-C_a(\xi-\eta)\big|\leq C\epsilon\,,
\end{displaymath}
we obtain that
\begin{displaymath}
   \Big\|N\Eb_f(\tilde{W}_2)-2\LL^{-1}(\beta a\nabla\bar{\phi})\Big\|_{1,2}\leq C\beta\epsilon \,.
\end{displaymath}
Let $W_0=W_0^1+W_0^2$. Combining the above with \eqref{eq:130}, \eqref{eq:143}, and
\eqref{eq:132} yields   
\begin{equation}
  \label{eq:133}
\Big\|N\Eb_f(W_0)-3\LL^{-1}(\beta a\nabla\bar{\phi})\Big\|_{1,2}\leq C\beta\epsilon^\alpha \,.
\end{equation}

It remains necessary, to bound the expectation of
\begin{displaymath}
  W_1=\int_{\partial B(\eta,\epsilon)} 
  a(\xi)G(x,\xi)\frac{\partial v_1}{\partial\nu}(\xi,\eta) \,ds_\xi  \,,
\end{displaymath}
since, obviously, $\phi_1=W_0+W_1$.  Let $\chi$ be given by \eqref{eq:25}
and set $\chi_\epsilon=\chi(|\xi-\eta|/\epsilon)$. Integration by parts yields
\begin{equation}
\label{eq:134}
  W_1=W_2+W_3\,,
\end{equation}
where
\begin{displaymath}
  W_2=\int_{\Omega\setminus B(\eta,\epsilon)} a(\xi)\chi_\epsilon \nabla_\xi G(x,\xi)\cdot\nabla v_1(\xi,\eta) \,d\xi\,,
\end{displaymath}
and
\begin{displaymath}
  W_3=\int_{\Omega\setminus B(\eta,\epsilon)} G(x,\xi) \nabla(a\chi_\epsilon)\cdot\nabla v_1(\xi,\eta) \,d\xi\,.
\end{displaymath}
It can now be easily verified, by interchanging the order of
integration, that
\begin{displaymath}
  \Eb_f(W_2)= \int_\Omega  a(\xi)\nabla_\xi G(x,\xi) \cdot\int_{\Omega\setminus B(\xi,\epsilon)}\chi_\epsilon(|\xi-\eta|) \nabla v_1(\xi,\eta)\,
  f_1(\eta)\,d\eta  \,d\xi= -\Delta^{-1}(\Div F_2) \,,
\end{displaymath}
where
\begin{displaymath}
  F_2(\xi) = a(\xi)\int_{\Omega_\epsilon\setminus B(\xi,\epsilon)}\chi_\epsilon(|\xi-\eta|) \nabla v_1(\xi,\eta)\,
  f_1(\eta)\,d\eta\,.
\end{displaymath}
It thus follows that 
\begin{equation}
\label{eq:135}
  \|N \nabla\Eb_f(W_2)\|_2 \leq \|NF_2\|_2\,.
\end{equation}
As, by .\eqref{eq:7}
\begin{displaymath}
  |F_2|^2 \leq C\epsilon^3  \int_{\Omega_\epsilon\setminus B(\xi,\epsilon)} |\nabla v_1(\xi,\eta)|^2\,d\eta\,,
\end{displaymath}
we obtain, interchanging once again the order of integration
\begin{displaymath}
   \|F_2\|^2_2 \leq C\epsilon^3  \int_{\Omega_\epsilon}  \int_{\Omega\setminus B(\eta,\epsilon)} |\nabla v_1(\xi,\eta)|^2\,d\xi\,d\eta\,.
\end{displaymath}
With the aid of \eqref{eq:34} we then obtain that
\begin{displaymath}
  \|F_2\|^2_2 \leq  C\epsilon^3\int_{\Omega_\epsilon} \Big(\epsilon^5 +\frac{\epsilon^6}{d(\eta,\partial\Omega)^3}\Big) \,d\eta \leq C\epsilon^7
\end{displaymath}
Hence, by \eqref{eq:135},
\begin{equation}
\label{eq:136}
  \|N \nabla\Eb_f(W_2)\|_2 \leq C\bar{\beta}\epsilon^{1/2} \,.
\end{equation}

Interchanging the order of integration once again yields
\begin{displaymath}
   \Eb_f(W_3)= \Delta^{-1}(F_3) \,,
\end{displaymath}
where
\begin{equation}
\label{eq:137}
  F_3(\xi) = a(\xi)\int_{\Omega_\epsilon\setminus B(\xi,\epsilon)} \nabla\chi_\epsilon(|\xi-\eta|)\cdot\nabla v_1(\xi,\eta)\,
  f_1(\eta)\,d\eta\,.
\end{equation}
As
\begin{equation}
\label{eq:138}
    \|\nabla\Eb_f(W_3)\|_2^2 \leq \|\Eb_f(W_3)\|_\infty\|F_3\|_1 \,,
\end{equation}
we seek an estimate for both $\|\Eb_f(W_3)\|_\infty$ and $\|F_3\|_1$.
To estimate the former we write
\begin{displaymath}
  \Eb_f(W_3)= w_{3,1}+w_{3,2}\,,
\end{displaymath}
where
\begin{displaymath}
  w_{3,1}=\int_{\Omega_{2\epsilon}} \int_{\Omega\setminus B(\eta,\epsilon)} G(x,\xi) \nabla(a\chi_\epsilon)\cdot\nabla
  v_1(\xi,\eta) \,d\xi f_1(\eta)\,d\eta\,,
\end{displaymath}
and
\begin{displaymath}
  w_{3,2}=\int_{\Omega_\epsilon\setminus\Omega_{2\epsilon}} \int_{\Omega\setminus B(\eta,\epsilon)} G(x,\xi) \nabla(a\chi_\epsilon)\cdot\nabla
  v_1(\xi,\eta) \,ds_\xi f_1(\eta)\,d\eta\,.
\end{displaymath}

Since whenever $d(\eta,\partial\Omega)\geq2\epsilon$ we have, by \eqref{eq:35},
\begin{displaymath}
  \int_{\Omega\setminus B(\eta,\epsilon)} \nabla(a\chi_\epsilon)\cdot\nabla
  v_1(\xi,\eta) \,d\xi = \int_{\partial B(\eta,\epsilon)} a(\xi)\frac{\partial v_1}{\partial\nu}(\xi,\eta)
  \,ds_\xi =0 \,,
\end{displaymath}
we may write
\begin{displaymath}
  w_{3,1}=\int_{\Omega_{2\epsilon}} \int_{A_\epsilon(\eta) } [G(x,\xi)-G(x,\eta)] \nabla(a\chi_\epsilon)\cdot\nabla
  v_1(\xi,\eta) \,d\xi f_1(\eta)\,d\eta\,,
\end{displaymath}
where $A_\epsilon(\eta)=(B(\eta,2\epsilon)\setminus B(\eta,\epsilon))\cap\Omega$. By \eqref{eq:142}, for every
$\xi\in A_\epsilon(\eta)$, we have 
\begin{displaymath}
  |G(x,\xi)-G(x,\eta)|\leq C\min \Big(
  \frac{1}{|x-\xi|},\frac{\epsilon}{|x-\xi|^2}\Big) \,.
\end{displaymath}
Hence,
\begin{displaymath}
  \int_{A_\epsilon(\eta)}|G(x,\xi)-G(x,\eta)|^2\,d\xi\leq C\frac{\epsilon^{5/2}}{|x-\eta|^2} \,.
\end{displaymath}
As a result, we obtain with the aid of \eqref{eq:34} that
\begin{multline*}
  |w_{3,1}(x)| \leq \frac{C}{\epsilon} \int_{\Omega_{2\epsilon}} \Big[\int_{A_\epsilon(\eta) } |G(x,\xi)-G(x,\eta)|^2
  d\xi \Big]^{1/2}\|\nabla v_1(\cdot,\eta)\|_{L^2(\Omega\setminus B(\eta,\epsilon))}\, d\eta \\\leq
  C\epsilon^{3/2}\int_{\Omega_{2\epsilon}} \frac{1}{|x-\eta|^2}
  \Big(\epsilon^{5/2}+\frac{\epsilon^3}{d(\eta,\partial\Omega)^{3/2}}\Big)\, d\eta\,.
\end{multline*}
It can be easily verified that
\begin{displaymath}
  \sup_{x\in\Omega}\int_{\Omega_{2\epsilon}} \frac{1}{|x-\eta|^2}\frac{1}{d(\eta,\partial\Omega)^{3/2}}\,
  d\eta\leq C|\ln \epsilon|\epsilon^{-1/2}\,.
\end{displaymath}
Consequently,
\begin{equation}
  \label{eq:139}
\|w_{3,1}\|_\infty\leq C\epsilon^4|\ln \epsilon| \,. 
\end{equation}

Next we estimate $w_{3,2}$. Here we use the fact that $G(x,\xi)=0$ for
all $\xi\in\partial\Omega$ to obtain, by \eqref{eq:142}, that
\begin{displaymath}
  G(x,\xi) \leq C\min\Big(\frac{1}{|x-\xi|},\frac{d(\xi,\partial\Omega)}{|x-\xi|^2}\Big)\,.
\end{displaymath}
Then, 
\begin{multline*}
    |w_{3,2}(x)| \leq \frac{C}{\epsilon} \int_{\Omega_{2\epsilon}} \Big[\int_{A_\epsilon(\eta) } |G(x,\xi)|^2
  d\xi \Big]^{1/2}\|\nabla v_1(\cdot,\eta)\|_{L^2(\Omega\setminus B(\eta,\epsilon))}\,d\eta
  \leq \\C\epsilon^{3/2}\int_{\Omega_\epsilon\setminus\Omega_{2\epsilon}} \frac{1}{|x-\eta|^2}
  \Big(\epsilon^{5/2}+\frac{\epsilon^3}{d(\eta,\partial\Omega)^{3/2}}\Big)\, d\eta\,.
\end{multline*}
From which we easily obtain that $\|w_{3,1}\|_\infty \leq C|\ln \epsilon|\epsilon^4$\,, and hence,
\begin{equation}
\label{eq:140}
   \|\Eb_f(W_3)\|_2 \leq C\epsilon^4|\ln \epsilon| \,.
\end{equation}
We now use \eqref{eq:137} to obtain that
\begin{displaymath}
  |F_3(\xi)| \leq \frac{C}{\epsilon}\int_{A_\epsilon(\xi)} |\nabla v_1(\xi,\eta)| f_1(\eta)\,d\eta\,.
\end{displaymath}
Integrating with respect to $\xi$ yields, after we interchange the
order of integration, with the aid of \eqref{eq:34},
\begin{displaymath}
  \|F_3\|_1 \leq C\epsilon^3\,.
\end{displaymath}
Combining the above with \eqref{eq:140} and \eqref{eq:138} yields
\begin{displaymath}
    \|N \nabla\Eb_f(W_3)\|_2 \leq C\bar{\beta}\epsilon^{1/2}|\ln \epsilon|^{1/2} \,,
\end{displaymath}
which together with \eqref{eq:136}, \eqref{eq:134} and \eqref{eq:130}
completes the proof of \eqref{eq:129}.
\end{proof}

\appendix

\section{Green's function properties}
\label{app:A}

Let $G$ denote the (positive) Green's function associated with the
Dirichlet realization in $\Omega$ of $\A=\Div A\nabla $, where
$A\in C^{1,\alpha}(\Omega,M^{3\times3})$ satisfies
\begin{displaymath}
  \lambda|\xi|^2\leq \xi\cdot A\xi\leq \Lambda|\xi|^2 \quad \forall (x,\xi)\in\Omega\times\R^3 
\end{displaymath}.
We now prove
\begin{lemma}
  There exists $C(\Omega)>0$ such that
  \begin{equation}
    \label{eq:141}
G(x,\xi)\leq\frac{C}{|x-\xi|} \quad \forall(x,\xi)\in\Omega^2\,.
  \end{equation}
\end{lemma}
\begin{proof}
  Let $x_0\in\Omega$ and $R>0$ be such $\Omega\Subset B(x_0,R)$. Let further
  \begin{displaymath}
    \tilde{A}=
    \begin{cases}
      A & x\in\overline{\Omega}  \\
      1 & x\in B(x_0,R)\setminus\overline{\Omega} \,.
    \end{cases}
  \end{displaymath}
Let $\tilde{G}$ denote the Green's function associated with the
Dirichlet realization in $B(x_0,R)$ of $\tilde{\A}=\Div \tilde{A}\nabla$. Since
$G(x,y) <\tilde{G}(x,y)$ for all $(x,y)\in\partial\Omega\times\Omega$, we obtain by the
maximum principle we that $G(x,y) <\tilde{G}(x,y)$ for all
$(x,y)\in\Omega\times\Omega$. 

Let $G_\Delta$ denote the Green's function associated with the Dirichlet
realization of $-\Delta$ in $B(x_0,R)$. By \cite[Theorem 7.1]{lietal63} we
have that
\begin{displaymath}
  \tilde{G}(x,y)\leq CG_\Delta\,,
\end{displaymath}
from which \eqref{eq:141} readily follows.
\end{proof}

We can now state
\begin{lemma}
    For every multi-index $\beta$, with $|\beta|\leq2, $there exists $C(\Omega,\beta)>0$ such that
  \begin{equation}
\label{eq:142}
|D^\beta G(x,\xi)|\leq\frac{C}{|x-\xi|^{|\alpha|}} \quad \forall(x,\xi)\in\Omega^2\,.
  \end{equation}
\end{lemma}
\begin{proof}
  We skip the proof, as it is almost identical with the proof of 
  \cite[Lemma A.2]{al13}.
\end{proof}

\bibliography{surfk}
 \end{document}